\documentclass[1996/10/24]{amsart}

\usepackage{color}

\allowdisplaybreaks[3]

\usepackage{amsfonts}
\usepackage{epsf,latexsym,amsfonts,amsbsy,mathrsfs}
\usepackage{amsmath, amssymb, amsthm,amscd,amsxtra}
\usepackage[]{graphicx,subfigure}
\usepackage{epstopdf}
\usepackage{float}
\usepackage{multirow}
\usepackage{bbm}
\usepackage{stmaryrd}
\usepackage{epsfig}
\usepackage{tikz}
\usepackage{dsfont}
\usetikzlibrary{positioning} 
\usepackage{xcolor}

\usepackage{amssymb,stmaryrd,color}
\usepackage{lipsum}
\usepackage{amsfonts}
\usepackage{graphicx}
\usepackage{epstopdf}
\usepackage{algorithmic}
\usepackage{bm}
\usepackage{multirow}

\ifpdf
  \DeclareGraphicsExtensions{.eps,.pdf,.png,.jpg}
\else
  \DeclareGraphicsExtensions{.eps}
\fi

\newtheorem{theorem}{Theorem}[section]
\newtheorem{lemma}[theorem]{Lemma}

\theoremstyle{definition}

\theoremstyle{remark}
\newtheorem{remark}[theorem]{Remark}

\theoremstyle{corollary}
\newtheorem{corollary}[theorem]{Corollary}

\numberwithin{equation}{section}

\newcommand{\ignore}[1]{}

\newcommand{\Bb}{{\boldsymbol{b}}}

\newcommand{\Bn}{{\boldsymbol{n}}}

\newcommand{\Bp}{{\boldsymbol{p}}}
\newcommand{\Bq}{{\boldsymbol{q}}}

\newcommand{\Bw}{{\boldsymbol{w}}}
\newcommand{\Bx}{{\boldsymbol{x}}}

\newcommand{\Bpsi}{{\boldsymbol{\psi}}}

\newcommand{\BB}{{\boldsymbol{B}}}

\newcommand{\BW}{{\boldsymbol{W}}}

\begin{document}

\title[Maximum-norm estimates of HDG methods for parabolic problem]{A unified analysis of maximum-norm estimates
for a class of HDG methods for parabolic problem in polyhedral domains}

\author{Huangxin Chen}
\address{School of Mathematical Sciences and Fujian Provincial Key Laboratory on Mathematical Modeling and
High Performance Scientific Computing, Xiamen University, Fujian, 361005, China}
\email{chx@xmu.edu.cn}

\author{Haitao Leng}
\address{School of Mathematic and Information Sciences, Guangzhou University, Guangzhou 510006, Guangdong, China}
\email{htleng@m.scnu.edu.cn}

\author{Weifeng Qiu}
\address{Department of Mathematics, City University of Hong Kong, Hung Hom, Hong Kong}
\email{weifeqiu@cityu.edu.hk}

\thanks{The third author is corresponding author.}

\subjclass[2010]{65N15, 65N30, 76M10}

\date{}

\begin{abstract}
This paper studies a general class of semi-discrete hybridizable discontinuous Galerkin (HDG) methods, including mixed methods, for parabolic problems in nonconvex polygonal and polyhedral domains. By developing local energy error estimates together with energy estimates for a regularized Green's function, we establish a unified framework to prove the maximum-norm stability of both the semigroup defined by the semi-discrete scheme and the corresponding discrete solutions. The stability analysis shows that the main challenges stem from the treatment of numerical flux variables and the inherent asymmetry of the discrete scheme. These challenges are intrinsic to numerical approaches formulated within the mixed framework for parabolic equations. The asymmetry prevents the direct application of the double kick-back argument, while the presence of flux variables requires special techniques to control their values at the initial time. We emphasize that the analytical tools developed here to address these challenges are sufficiently general to be adapted for maximum-norm stability investigations of a broader class of numerical schemes arising from mixed formulations of parabolic equations.
Furthermore, by the stability results and their proofs, we derive the maximal regularity of the semi-discrete solution in $L^{\infty}((0,T);L^p(\Omega))$-norm and reduce the maximum-norm error estimates to those of the corresponding elliptic equations and the $L^2$-orthogonal projection. Since the derivation of the local energy error estimates does not rely on the lifting operator, which is only available for simplicial meshes, our results (excluding mixed methods) remain valid for polygonal/polyhedral meshes. Note that the aforementioned results are obtained for the first time to numerical methods based on the mixed formulation, and the proof techniques can be applied to the analysis of other numerical methods derived from the mixed formulation.
\end{abstract}
\keywords{maximum-norm stability, maximum-norm error estimates, HDG methods, polyhedral meshes, local energy error estimates}

\maketitle

\section{Introduction}
Let $\Omega\subset \mathbb{R}^d$ be an open bounded domain that is nonconvex with polygonal ($d=2$) or polyhedral ($d=3$) Lipschitz boundary $\partial\Omega$,
in which we consider the following parabolic equations
\begin{subequations}\label{model}
\begin{align}
\partial_t u(x,t)-\Delta u(x,t)=&f(x,t),\quad \text{in}~\Omega\times (0,T],\\
u(x,t)=&0,\qquad  \quad \  \text{on}~\partial\Omega\times [0,T],\\
u(x,0)=&u_0(x),\quad \ \, \text{in}~\Omega,
\end{align}
\end{subequations}
where $T$ is a positive constant.

If we define the elliptic operator $-\Delta: H_0^1(\Omega)\rightarrow H^{-1}(\Omega)$ by
\begin{align*}
(-\Delta w,v)=(\nabla w,\nabla v),\quad \forall w,v\in H_0^1(\Omega),
\end{align*}
the solution of (\ref{model}) for the homogeneous case of $f=0$ can be expressed as $u(t)= E(t)u_0$, where $\{E(t)=e^{t\Delta}\}_{t>0}$ is the semigroup
generated by the operator $-\Delta$. By the theory of parabolic equations, we know that $\{E(t)\}_{t>0}$ is an analytic semigroup on $\mathcal{C}_0(\overline{\Omega})$ and $L^p(\Omega)~(1\leq p<\infty)$, satisfying the following estimates (cf. \cite{tmw2006}):
\begin{align}
\|E(t)v\|_{L^p(\Omega)}+t\|\partial_t E(t)v\|_{L^p(\Omega)}\leq C\|v\|_{L^p(\Omega)},\quad \forall v\in L^p(\Omega),~t>0,\label{msfs:1}\\
\|E(t)v\|_{L^{\infty}(\Omega)}+t\|\partial_tE(t)v\|_{L^{\infty}(\Omega)}\leq C\|v\|_{L^{\infty}(\Omega)},\quad \forall v\in \label{msfs:2} \mathcal{C}_0(\overline{\Omega}),~t>0,
\end{align}
where $\mathcal{C}_0(\overline{\Omega})$ is the space of continuous functions with homogeneous boundary.

It is well-known that the discontinuous Galerkin (DG) method is very flexible when it is applied to solve partial differential equations. However, the final discrete system still involves a large number of globally coupled degrees of freedom. In order to overcome this issue,
a new class of HDG methods based on the mixed formulation, which retains the advantage of DG methods and results in a system with significantly reduced
degrees of freedom, is proposed by Cockburn et al. \cite{cgl2009} for second-order elliptic equations. Currently, HDG methods have been successfully used in engineering and scientific computing, see \cite{ccq2017, cqs2018, cfq2018, cgnps2011, lc2024} and the references therein for more detail.

In this paper, our goal is to develop a unified framework for three different HDG methods (including mixed methods, see (\ref{HDG})) to establish the discrete analogue of (\ref{msfs:2}) in nonconvex polygonal and polyhedral domains. Then, based on the result and its proofs, we further prove the maximum-norm stability, the maximal regularity in the $L^{\infty}((0,T);L^p(\Omega))$-norm and the maximum-norm error estimate for the semi-discrete solution. Note that these HDG methods are based on the mixed formulation of the problem and
their semi-discrete solution is $(\Bq_h,u_h,\widehat{u}_h)$, where $\Bq_h$ approximates the flux $-\nabla u$ and $(u_h,\widehat{u}_h)$ approximates the exact solution $u$ in $\Omega$ and mesh interfaces. Roughly speaking, in this paper we prove the following estimation results
\begin{align}
&\Vert u_{h} \Vert_{L^{\infty}(\Omega)}
+ t\Vert \partial_{t}u_{h} \Vert_{L^{\infty}(\Omega)}
\leq C \Vert u_{0,h}\Vert_{L^{\infty}(\Omega)},
\qquad \forall t \in (0,T],~\text{if}~f=0,\label{main_result_stability2}\\
&\Vert u_{h} \Vert_{L^{\infty}(\Omega)} \leq
C \left( \Vert f \Vert_{L^{\infty}((0,t); L^{\infty}(\Omega))}
+ \Vert u_{0,h}\Vert_{L^{\infty}(\Omega)} \right),
\qquad \forall t \in (0,T],\label{main_result_stability1}
\end{align}
and
\begin{align}
&\|\partial_t u_h\|_{L^{\infty}((0,T);L^p(\Omega))}+\|\Delta_h u_h\|_{L^{\infty}((0,T);L^p(\Omega))}\nonumber\\
&\qquad \leq C|\ln h|\|f\|_{L^{\infty}((0,T);L^p(\Omega))},\quad  1\leq p\leq \infty,
~\text{if}~u_0=0,\label{main_result_stability3}\\
&\|u-u_h\|_{L^{\infty}((0,T);L^{\infty}(\Omega))}\leq C|\ln h|\big(\|u-u_h^e\|_{L^{\infty}((0,T);L^{\infty}(\Omega))}\nonumber\\
& \qquad +\|u-\Pi_Vu\|_{L^{\infty}((0,T);L^{\infty}(\Omega))}
+\|\Pi_Vu(0)-u_h(0)\|_{L^{\infty}(\Omega)}\big),\label{main_result_stability4}
\end{align}
where $u_{0,h}\in V_h$ is an approximation of $u_0$, $V_h$ is the discrete space provided in Section \ref{sec2}, $\Delta_h$ is the discrete Laplacian operator defined in Section \ref{sec5}, $h$ is the mesh size, $\Pi_V$ is the $L^2$-orthogonal projection onto $V_h$, and
$u_h^e$ is the HDG solution of the elliptic equations defined in (\ref{HDG-E}).
The constants $C$ in (\ref{main_result_stability2}), (\ref{main_result_stability1}), (\ref{main_result_stability3}) and (\ref{main_result_stability4}) are independent of $T$, which implies that these results also hold
as $T$ tends to infinity. Detailed descriptions are given in Theorem \ref{max-stability} for (\ref{main_result_stability2}) and
(\ref{main_result_stability1}), Theorem \ref{stability} for (\ref{main_result_stability3}), and Corollary \ref{max-ee} for
(\ref{main_result_stability4}). To the best of our knowledge, there were no similar results  for
numerical methods based on the mixed formulation. Furthermore, the analytical techniques established
in this work for (\ref{main_result_stability2})-(\ref{main_result_stability4}) are sufficiently general and can be adapted to the similar analysis of a broader class of numerical schemes arising from mixed formulations of parabolic equations.

If $\Omega$ is a smooth domain, the authors in \cite{stw1980, stw1998, tw2000, l2015, kk2020} have shown (\ref{main_result_stability2}) for
semi-discrete finite element methods (FEMs). If $\Omega$ is a convex domain, (\ref{main_result_stability2}) is proved in \cite{r1990, ls2017} for semi-discrete FEMs and in \cite{clq2025} for semi-discrete interior penalty hybridized discontinuous Galerkin (IP-HDG) methods. In \cite{stw1980, r1990}, the result includes a logarithmic factor $|\ln h|$ and is restricted to the case of $d=2$. In \cite{stw1998, tw2000, l2015, kk2020, ls2017, clq2025}, the logarithmic factor is removed and the result is extended to the case of $d=3$ by using a local energy error estimate and energy estimates of the Green's and the regularized Green's functions. Recently, the results shown in \cite{ls2017} are further extended to the case in which the domain is nonconvex \cite{l2019}.
In a curvilinear polyhedral domain that can not be exactly triangulated, in order to address the domain perturbation effect the authors in  \cite{qx2026} use isoparametric finite element methods to approximate the parabolic equations and establish the stability result (\ref{main_result_stability2}).
As for (\ref{main_result_stability3}) and (\ref{main_result_stability4}), they have been investigated for semi-discrete FEMs of parabolic equations in smooth domains \cite{ge2006, kk2020, stw1998, l2015}, convex domains \cite{ls2017}, and nonconvex domains
\cite{l2019}. Note that (\ref{main_result_stability3}) and (\ref{main_result_stability4}) have also been proved for isoparametric finite element methods in \cite{qx2026}. Since the result (\ref{main_result_stability2}) serves as the foundation for deriving other results in this paper, here we only present a brief review on (\ref{main_result_stability3}) and (\ref{main_result_stability4}). Moreover, up to now, no results similar to (\ref{main_result_stability1}) have been identified in the literature. For interested readers, we additionally refer to \cite{cmz2021, cc2004, dlsw2011, g2006, glrs2009, lc2023, lv2016, sw1995} for maximum-norm error estimates of elliptic equations, \cite{lv2016a, lv2017} for pointwise best approximations of the fully discrete scheme for parabolic equations, and \cite{bll2025} for the weak maximum principle of FEMs for parabolic equations.

It is worth noting that the semi-discrete scheme presented in \cite{stw1980, r1990, stw1980, stw1998, tw2000, l2015, kk2020, ls2017, clq2025, l2019, qx2026}, which is based on the primary formulation of the problem and uses $\nabla u_h$ to approximate $\nabla u$, is symmetric. In contrast, the HDG method of this paper based on the mixed formulation of the problem is asymmetric since $\Bq_h$ is not exactly the gradient of $u_h$ on each element. These differences lead to two difficulties specific to numerical methods based on the mixed formulation of parabolic equations, which we will encounter in the proof of (\ref{main_result_stability2}).
First, in the local energy error estimate for $\partial_{t}e_u$ (cf. Lemma~\ref{local_energy_estimate2}), terms involving $\partial_{t}e_{\boldsymbol{q}}$ and $\partial_t
(\mathcal{L}e_u-e_{\widehat{u}})$ along mesh interfaces arise in our analysis. Notably, such terms are absent from
\cite[(A.8)]{l2019} for FEMs and from \cite[Lemma~$3.7$]{clq2025} for IP-HDG methods.
These terms arise from the asymmetry of the discrete scheme and will prevent the successful application of a key technique called the double kick-back argument. Second, there is no explicit formulation in this article to represent $(\boldsymbol{q}_{h},\widehat{u}_{h})$
in terms of $u_{h}$. Therefore, even if the numerical initial data $u_{0,h}$ decays very fast, it is not clear how to show the decay properties of the corresponding $(\boldsymbol{q}_{h},\widehat{u}_{h})$ which is essential in maximum-norm stability analysis.
Since the semi-discrete scheme in \cite{stw1980, r1990, stw1980, stw1998, tw2000, l2015, kk2020, ls2017, clq2025, l2019, qx2026} based on the primary formulation is symmetric and does not involve the term $\Bq_h$, the two difficulties mentioned above can be both avoided.
Furthermore, when the numerical method under consideration is applied to elliptic equations, the aforementioned difficulties are absent in the context of maximum-norm stability analysis.

Next, we give a brief description on the new technique we developed in this article
to overcome the challenges mentioned above. In the proof of
Lemma~\ref{local_energy_estimate3}, we repeat the local energy error estimate obtained in Lemma \ref{local_energy_estimate1} between two space-time sub-domains $I_0\times\Omega_0$ and $I_1\times \Omega_1$,  such that there are two factors
$(h\rho^{-1})^{\beta_{1}}$ and $(h\rho^{-1})^{\beta_2}$ appeared in front of the terms involving $\partial_{t}
e_{\boldsymbol{q}}$ and $\boldsymbol{q}-\boldsymbol{q}_h$ respectively. Here the positive real numbers $\beta_1$ and $\beta_{2}$ depend on
the total number of repeating of local energy error estimates, and $\rho$ is the distance between sub-domains $\Omega_0$ and $\Omega_1$.
When Lemma~\ref{local_energy_estimate3} is applied in the double kick-back
argument (see the proof of Theorem~\ref{max-stability}), we choose $\beta_1=\beta_{2}=4+\frac{d}{2}$ such that corresponding terms can be absorbed properly.
In Lemma~\ref{ee3}, we use a similar argument to prove a local energy estimate
for $(\mathcal{R}_h,\Gamma_h,\widehat{\Gamma}_h)$ at $t=0$, where $(\mathcal{R}_h,\Gamma_h,\widehat{\Gamma}_h)$ is the semi-discrete HDG solution of the regularized Green's function. And this will be used to control the value of $(\mathcal{R}_h,\Gamma_h,\widehat{\Gamma}_h)$ at $t=0$ in the proof of Theorem \ref{max-stability}.

In the literature, we have found only one paper \cite{l2019}, based on the primary formulation, that investigates (\ref{main_result_stability2}), (\ref{main_result_stability3}) and (\ref{main_result_stability4}) for parabolic equations in nonconvex polygonal/polyhedral domains. In addition, compare with \cite{ls2017, l2019, qx2026}, our proofs have not used the superapproximation due to the special choice of cut-off functions. Since the lifting operator that is only available
for simplicial meshes \cite{lc2023} is not utilized in this paper, our results (excluding mixed methods) are also valid for polygonal/polyhedral meshes.

The rest of this paper is arranged as follows. In Section \ref{sec2}, notation, mesh and the HDG semi-discrete scheme for parabolic problem (\ref{model}) are
provided. In Section \ref{sec3}, we establish a local energy error estimate that will be used to prove the maximum-norm stability. In
Section \ref{sec4}, we prove energy estimates of the regularized Green's function and use these together with the local energy error estimate obtained in
Section \ref{sec3} to prove the maximum-norm stabilities of the discrete semigroup and discrete solutions. In Section \ref{sec5}, we prove the maximal regularity and maximum-norm error estimates by using the stability results established in Section \ref{sec4}.
Finally, we end this paper with some conclusions in Section \ref{sec6}.

\section{Notation and HDG semidiscretization}\label{sec2}
For any nonnegative integer $s$, an open subset $D\subset\Omega$ and $1\leq p\leq \infty$, let $W^{s,p}(D)$ be the Sobolev spaces \cite{a1975} with
norms
\begin{align*}
\|v\|_{W^{s,p}(D)}^p=\sum_{i=0}^s|v|_{W^{i,p}(D)}^p\quad 1\leq p<\infty,\quad \|v\|_{W^{s,\infty}(D)}=\max_{0\leq i\leq s}|v|_{W^{i,\infty}(D)},
\end{align*}
and seminorms
\begin{align*}
|v|^p_{W^{i,p}(D)}=\sum_{|\alpha|=i}\int_D\left|\frac{\partial^{\alpha}v}{\partial x^{\alpha}}\right|^p dx\quad 1\leq p<\infty,\quad
|v|_{W^{i,\infty}(D)}=\max_{|\alpha|=i}\left\|\frac{\partial^{\alpha}v}{\partial x^{\alpha}}\right\|_{L^{\infty}(D)},
\end{align*}
where
\begin{align*}
\frac{\partial^{\alpha}v}{\partial x^{\alpha}}=\frac{\partial^{|\alpha|}v}{\partial x_1^{\alpha_1}\cdots\partial x_d^{\alpha_d}}
\end{align*}
for the multi-index $\alpha=(\alpha_1,\cdots,\alpha_d),~\alpha_1\geq 0,~\cdots,~\alpha_d\geq 0$, and $|\alpha|=\alpha_1+\cdots+\alpha_d$.
Denote by $W_0^{s,p}(D)~(1<p<\infty)$ the completion of $\mathcal{C}_0^{\infty}(D)$ according to the norm $\|\cdot\|_{W^{s,p}(D)}$, where
$\mathcal{C}_0^{\infty}(D)$ is the space of functions with continuous derivatives of arbitrary order and compact support in $D$. If $p=2$, the
Sobolev spaces $W^{s,2}(D)$ are denoted by $H^s(D)$ with norms $\|\cdot\|_{H^s(D)}$ and seminorms $|\cdot|_{H^s(D)}$. If $s=0$, we set $L^p(D)=W^{0,p}(D)$
with norms $\|\cdot\|_{L^p(D)}$.

For any Banach space $X$ and a given $T>0$, $L^p((0,T);X)$ and $W^{1,p}((0,T);X)$ denote the Bochner spaces \cite{y1980} with norms
\begin{align*}
\|v\|_{L^p((0,T);X)}^p=&\int_0^T\|v(t)\|_X^p dt\quad 1\leq p<\infty,\quad \|v\|_{L^{\infty}((0,T);X)}=\max_{t\in(0,T)}\|v(t)\|_X,\\
&\|v\|_{W^{1,p}((0,T);X)}=\|v\|_{L^p((0,T);X)}+\|\partial_t v\|_{L^p((0,T);X)}.
\end{align*}
Then the variational formulation of parabolic problem (\ref{model}) reads as follows: Find $u\in L^2((0,T);H_0^1(\Omega))$ such that
\begin{align}
\label{weak}
\iint_{Q_T}-u\partial_t v+\nabla u\cdot\nabla v ~dxdt=\iint_{Q_T} fv ~dxdt+\int_{\Omega}u_0 v(x,0)~dx,
\end{align}
for any $v\in L^2((0,T);H_0^1(\Omega))\cap H^1((0,T);L^2(\Omega))$ with $v(T)=0$, where $Q_T=(0,T)\times \Omega$.

\begin{theorem}
\label{regularity}
(i) For $f\in L^2((0,T);L^2(\Omega))$ and $u_0\in L^2(\Omega)$, the variational formulation (\ref{weak}) has a unique solution
$u\in L^2((0,T);H_0^1(\Omega))$ with $\partial_t u\in L^2((0,T);H^{-1}(\Omega))$, where $H^{-1}(\Omega)$ is the dual space of $H_0^1(\Omega)$.
Moreover we have $u\in \mathcal{C}([0,T];L^2(\Omega))$.

(ii) If we further have $u_0\in H_0^1(\Omega)$, then the solution of (\ref{weak}) belongs to $L^2((0,T);H^{1+\gamma}(\Omega))\cap L^{\infty}((0,T);H_0^1(\Omega))$ with $\partial_tu\in L^2((0,T);L^2(\Omega))$, where $\gamma\in (\frac{1}{2},1)$ is a positive constant depending on the domain $\Omega$.
\end{theorem}

\begin{proof}
The results presented in (i) can be found in \cite[Theorem 3.9, Theorem 3.10 and Theorem 3.12]{t2010}. As for the results in (ii), they have been proved in
\cite[Theorem 5 in Chapter 7]{e2010}.
\end{proof}

To describe the semi-discrete scheme, let $\mathcal{T}_h$ be a conforming partition of the domain $\Omega$. Denote by $\mathcal{E}_h^o$ the set of all interior edges/faces of $\mathcal{T}_h$ and $\mathcal{E}_h^{\partial}$ the set of all boundary edges/faces.
We define $\mathcal{E}_h=\mathcal{E}_h^o\cup\mathcal{E}_h^{\partial}$ and $\partial\mathcal{T}_h=\{\partial K: K\in\mathcal{T}_h\}$, where $\partial K$
denotes the boundary of the element $K$. For any $K\in\mathcal{T}_h$ and $F\in\mathcal{E}_h$, let $h_K$ and $h_F$ be the diameters of $K$ and $F$.
We set $h=\max_{K\in\mathcal{T}_h}h_K$ and assume that the partition $\mathcal{T}_h$ is quasi-uniform.

Next, we introduce mesh-dependent Sobolev spaces
\begin{align*}
W_h^{s,p}(D)=\{v: v\in W^{s,p}(K\cap D),~\forall K\in\mathcal{T}_h\},
\end{align*}
with norms
\begin{align*}
\|v\|_{W_h^{s,p}(D)}^p=\sum_{i=0}^s|v|_{W_h^{i,p}(D)}^p\quad 1\leq p<\infty,\quad \|v\|_{W_h^{s,\infty}(D)}=\max_{0\leq i\leq s}|v|_{W_h^{i,\infty}(D)},
\end{align*}
and seminorms
\begin{align*}
|v|_{W_h^{i,p}(D)}^p=\sum_{K\in\mathcal{T}_h}|v|_{W^{i,p}(K\cap D)}^p\quad 1\leq p<\infty, \quad |v|_{W_h^{i,\infty}(D)}=\max_{K\in\mathcal{T}_h}|v|_{W^{i,\infty}(K\cap D)}.
\end{align*}
If $p=2$, the mesh-dependent Sobolev spaces $W_h^{s,2}(D)$ are denoted by $H_h^s(D)$ with norms $\|\cdot\|_{H_h^s(D)}$ and seminorms $|\cdot|_{H^s_h(D)}$.
If $s=0$, the mesh-dependent Sobolev spaces $W_h^{0,p}(D)$ are abbreviated by $L_h^p(D)$ with norms $\|\cdot\|_{L_h^p(D)}$.

Based on the partition $\mathcal{T}_h$, we introduce discontinuous finite element spaces as follows:
\begin{align*}
\BW_h=&\{\Bw_h\in (L^2(\Omega))^d: \Bw_h|_K\in \BW(K),~\forall K\in\mathcal{T}_h\},\\
V_h=&\{v\in L^2(\Omega): v|_K\in V(K),~K\in\mathcal{T}_h\},\\
\widehat{V}_h=&\{\mu\in L^2(\mathcal{E}_h): \mu|_F\in\mathcal{P}^k(F)~F\in\mathcal{E}_h,~\mu=0~\text{on}~\mathcal{E}_h^{\partial}\},
\end{align*}
for $k\geq 1$, where $\mathcal{P}^s(D)$ denotes the set of polynomials of degree no larger than $s$ on the domain $D$ and $\BW(K)$ and $V(K)$ are two local spaces that will be decided later. By setting $\Bq=-\nabla u$, we arrive at the mixed formulation of problem (\ref{model}):
\begin{subequations}\label{model-mixed}
\begin{align}
\Bq+\nabla u=&0,\quad \text{in}~\Omega\times(0,T],\\
\partial_t u+\nabla\cdot\Bq=&f,\quad \text{in}~\Omega\times (0,T],\\
u=&0,\quad \text{on}~\partial\Omega\times [0,T],\\
u|_{t=0}=&u_0,\quad \text{in}~\Omega.
\end{align}
\end{subequations}
Then the semi-discrete scheme of (\ref{model-mixed}) reads as follows: For $t\in [0,T]$, find $(\Bq_h(t),u_h(t),\widehat{u}_h(t))\in \BW_h\times V_h\times \widehat{V}_h$ such that
\begin{subequations}\label{HDG}
\begin{align}
(\Bq_h,\Bw_h)_{\mathcal{T}_h}-(u_h,\nabla\cdot\Bw_h)_{\mathcal{T}_h}+\langle\widehat{u}_h,\Bw_h\cdot\Bn\rangle_{\partial
\mathcal{T}_h}=0&,\label{HDG:1}\\
(\partial_t u_h,v_h)_{\mathcal{T}_h}-(\Bq_h,\nabla v_h)_{\mathcal{T}_h}+\langle\widehat{\Bq}_h\cdot\Bn,v_h\rangle_{\partial\mathcal{T}_h}=(f,v_h)_{\mathcal{T}_h}&,\label{HDG:2}\\
\langle\widehat{\Bq}_h\cdot\Bn,\widehat{v}_h\rangle_{\partial\mathcal{T}_h}=0&,\label{HDG:3}\\
\widehat{\Bq}_h\cdot\Bn=\Bq_h\cdot\Bn+\tau(\mathcal{L}u_h-\widehat{u}_h),\quad \text{on}~\partial\mathcal{T}_h&,\label{HDG:4}
\end{align}
\end{subequations}
for any $(\Bw_h,v_h,\widehat{v}_h)\in \BW_h\times V_h\times \widehat{V}_h$ and $u_h|_{t=0}=u_{0,h}$,
where $u_{0,h}\in V_h$ is an approximation of $u_0$, $\Bn$ is the outward unit normal vector to $\partial K$, $(\cdot,\cdot)_K$ and $\langle\cdot,\cdot\rangle_{\partial K}$ denote
the inner products in $L^2(K)$ and $L^2(\partial K)$, and
\begin{align*}
(\cdot,\cdot)_{\mathfrak{D}}=\sum_{K\in\mathfrak{D}}(\cdot,\cdot)_K,\quad \langle\cdot,\cdot\rangle_{\partial \mathfrak{D}}=\sum_{K\in \mathfrak{D}}
\langle\cdot,\cdot\rangle_{\partial K},
\end{align*}
for any $\mathfrak{D}\subset\mathcal{T}_h$.

In this paper, the local spaces $\BW(K)$ and $V(K)$ and the operator $\mathcal{L}$ are set as
\begin{align}\label{sp}
\left(\BW(K), V(K), \mathcal{L}\right)=\begin{cases}
\left((\mathcal{P}^k(K))^d+\Bx\mathcal{P}^k(K), \mathcal{P}^k(K), \mathcal{I}\right),\quad &\text{hybrid mixed DG},\\
\left((\mathcal{P}^k(K))^d, \mathcal{P}^k(K), \mathcal{I}\right), \quad &\text{HDG-1}, \\
\left((\mathcal{P}^k(K))^d, \mathcal{P}^{k+1}(K),\Pi_{\widehat{V}}\right), \quad &\text{HDG-2},
\end{cases}
\end{align}
where $\mathcal{I}$ is the identity mapping and $\Pi_{\widehat{V}}$ denotes the standard $L^2$-orthogonal projection onto $\mathcal{P}^k(F)$ on each
$F\in\mathcal{E}_h$. Here $\tau$ denotes the stabilization parameter. On each $K\in\mathcal{T}_h$, it is chosen as
\begin{align}
\begin{cases}
\tau=0~\text{on}~\partial K,\quad &\text{hybrid mixed DG},\\
\tau=h_K^{-1}~\text{on}~F_K^{\ast}~\text{and}~\tau=0~\text{on}~\partial K\backslash F_K^{\ast},\quad &\text{HDG-1},\\
\tau=\widetilde{\alpha} h_K^{-1}~\text{on}~\partial K,\quad &\text{HDG-2},
\end{cases}
\end{align}
where $\widetilde{\alpha}$ is a mesh independent and positive constant on all edges/faces and $F_K^{\ast}$ denotes a single edge/face of $\partial K$. For hybrid mixed DG methods \cite{es2010}, the partition $\mathcal{T}_h$ consists of shape-regular simplices. For HDG-1 \cite{cfq2018} and HDG-2 \cite{qs2016} methods, the partition $\mathcal{T}_h$ can be composed of polygonal/polyhedral elements which consist of triangles/tetrahedras. Moreover, for each polygonal/polyhedral element $K$, we assume that its sub-triangulation is shape-regular and quasi-uniform and the number of sub-elements is uniformly bounded by a positive constant $\mathcal{N}_d$.

\section{Local energy error estimates}\label{sec3}
This section is devoted to proving a local energy error estimate which plays an important role in the analysis of maximum-norm stability and maximum-norm error estimates. To the best of our knowledge, there is no such result in the literature for the semi-discrete scheme (\ref{HDG}) for parabolic problem, thus we
sketch a proof here. Before this, we provide some approximation properties that will be used frequently in the subsequent proofs.

\begin{lemma}(Inverse properties, \cite[Lemma 4.5.3]{bs2008}).\label{inv}
Let $1\leq q\leq p\leq\infty$ and $0\leq m\leq s\leq l+1$.

(i) If $v\in \{v\in L^2(\Omega):v|_K\in\mathcal{P}^l(K),~K\in\mathcal{T}_h\}$, then for $K\in\mathcal{T}_h$
\begin{align}
\|v\|_{W^{s,p}(K)}\leq Ch^{m-s+\frac{d}{p}-\frac{d}{q}}\|v\|_{W^{m,q}(K)}.
\end{align}

(ii) If $\Omega_0\subset\Omega_1\subset\Omega$ with $\text{dist}(\partial \Omega_0,\partial \Omega_1\backslash\partial\Omega)\geq rh $
for some $r>1$, then for $v\in \{v\in L^2(\Omega):v|_K\in\mathcal{P}^l(K),~K\in\mathcal{T}_h\}$,
\begin{align}
\|v\|_{W_h^{s,p}(\Omega_0)}\leq Ch^{m-s+\frac{d}{p}-\frac{d}{q}}\|v\|_{W_h^{m,q}(\Omega_1)}.
\end{align}
\end{lemma}

\begin{lemma}(Trace properties).\label{tra}
For any $K\in\mathcal{T}_h$ and $1\leq p\leq \infty$, we have
\begin{align}
\|v\|_{L^p(\partial K)}\leq C(h_K^{-\frac{1}{p}}\|v\|_{L^p(K)}+h_K^{1-\frac{1}{p}}\|\nabla v\|_{L^p(K)}),\quad \forall v\in W^{1,p}(K),\label{tra:1}
\end{align}
and
\begin{align}
\|v\|_{L^2(\partial K)}\leq C\left(h_K^{-\frac{1}{2}}\|v\|_{L^2(K)}+h_K^{\gamma-\frac{1}{2}}|v|_{H^{\gamma}(K)}\right),\quad  \forall v\in H^{\gamma}(K), \label{tra:2}
\end{align}
where $|\cdot|_{H^{\gamma}(K)}$ is the seminorm of $H^{\gamma}(K)$ defined as
\begin{align*}
|v|_{H^{\gamma}(K)}^2=\int_K\int_K\frac{|v(x)-v(y)|^2}{|x-y|^{d+2\gamma}}dxdy.
\end{align*}
Note that the positive constant $C$ appeared in (\ref{tra:2}) depends on $\gamma$.
\end{lemma}
\begin{proof}
The result (\ref{tra:1}) can be found in \cite[(2.3)]{g2006}. In order to prove (\ref{tra:2}), for each polygonal/polyhedral element $K$, let $\{K_i\}_{i=1}^{N_{K,d}}$ denote the set of sub-elements that intersect $\partial K$, and set $e_i=\partial K\cap\partial K_i~(i=1,\cdots,N_{K,d})$. Let $\widehat{\mathcal{K}}$ be the reference element of $K_i$, and let the affine mapping be given by
$\overline{K}_i\ni\Bx=\BB_{K_i}\widehat{\Bx}+\Bb_{K_i}$. Then, we define $\widehat{v}_i(\widehat{\Bx})=v_i(\BB_{K_i}\widehat{\Bx}+\Bb_{K_i})=v_i(\Bx)$ and have
\begin{align*}
\|v_i\|_{L^2(e_i)}\leq Ch_K^{\frac{d-1}{2}}\|\widehat{v}_i\|_{L^2(\widehat{e}_i)},\quad \|\widehat{v}_i\|_{L^2(\widehat{\mathcal{K}})}\leq Ch_K^{-\frac{d}{2}}
\|v_i\|_{L^2(K_i)},\quad |\widehat{v}_i|_{H^{\gamma}(\widehat{\mathcal{K}})}\leq Ch_K^{-\frac{d}{2}+\gamma}|v_i|_{H^{\gamma}(K_i)},
\end{align*}
for any $i=1,\cdots,N_{K,d}$, where $v_i$ is the restriction of $v$ on $K_i$.
Then according to the standard trace inequality in the reference element $\widehat{\mathcal{K}}$, we obtain
\begin{align*}
\|v\|^2_{L^2(\partial K)}=&\sum_{i=1}^{N_{K,d}}\|v_i\|^2_{L^2(e_i)}\leq Ch_K^{d-1}\sum_{i=1}^{N_{K,d}}\|\widehat{v}_i\|^2_{L^2(\widehat{e}_i)}\\
\leq&Ch_K^{d-1}\sum_{i=1}^{N_{K,d}}\left(\|\widehat{v}_i\|^2_{L^2(\widehat{\mathcal{K}})}+|\widehat{v}_i|^2_{H^{\gamma}(\widehat{\mathcal{K}})}\right)\quad ( C~\text{depends on}~\gamma)\\
\leq& C\sum_{i=1}^{N_{K,d}}\left(h_K^{-1}\|v_i\|^2_{L^2(K_i)}+h_K^{2\gamma-1}|v_i|^2_{H^{\gamma}(K_i)}\right)\\
\leq& C\left(h_K^{-1}\|v\|^2_{L^2(K)}+h_K^{2\gamma-1}|v|^2_{H^{\gamma}(K)}\right).
\end{align*}
We conclude the proof.
\end{proof}

With all ingredients at hand, now we follow \cite{clq2025, lc2023, ls2017, l2019, tw2000} to prove a local energy error estimate.
In the rest of this section, we let $\Omega_0\subset\Omega_2\subset\Omega_1\subset\Omega$ with $\rho$ satisfying
$$
\rho=\text{dist}(\partial\Omega_0,\partial\Omega_2\backslash\partial\Omega)=
\text{dist}(\partial\Omega_2,\partial\Omega_1\backslash\partial\Omega)\geq rh,
$$
and $I_0 \subset I_2\subset I_1\subset (0,T)$ with $\rho^2=\text{dist}(\partial I_0,\partial I_2\backslash\{0,T\})=\text{dist}(\partial I_2,
\partial I_1\backslash\{0,T\})$, where $r>1$ is a sufficiently large fixed constant. Moreover, we
let $\varphi_1\in\mathcal{C}^{\infty}(\Omega)$ satisfy $\varphi_1= 1$ on $\Omega_0$, $\varphi_1= 0$ on $\Omega\backslash\Omega_2$, and $|\varphi_1|_{W^{l,\infty}(\Omega)}\leq C\rho^{-l}$
for $l=0,1,\cdots$. Also let $\varphi_2\in\mathcal{C}^{\infty}(0,T)$ satisfy $\varphi_2=1$ on $I_0$, $\varphi_2= 0$ on $(0,T)\backslash I_2$,  and $|\frac{d \varphi_2}{dt}|\leq C\rho^{-2}$.

Based on the function $\varphi_1$, for each $K\in\mathcal{T}_h$ and $F\in\mathcal{E}_h$ we define $(\varphi_{1,h}|_K, \widehat{\varphi}_{1,h}|_F)\in\mathcal{P}^0(K)\times \mathcal{P}^0(F)$ as
\begin{align*}
\varphi_{1,h}|_K=\sup_{x\in K}\varphi_1(x),\quad \widehat{\varphi}_{1,h}|_F=\sup_{x\in F}\varphi_{1}(x).
\end{align*}
It is easy to observe that the defined functions $\varphi_{1,h}$ and $\widehat{\varphi}_{1,h}$ satisfy the following estimates
\begin{align}
\|\varphi_1-\varphi_{1,h}\|_{L^{\infty}(K)}\leq Ch\rho^{-1},\quad \|\varphi_1-\widehat{\varphi}_{1,h}\|_{L^{\infty}(F)}\leq Ch\rho^{-1}.\label{supera}
\end{align}

By defining
\begin{align*}
& e_{\Bq} = \Pi_W\Bq - \Bq_{h}, \quad e_{u} = \Pi_V u - u_{h}, \quad e_{\widehat{u}} = \Pi_{\widehat{V}}u - \widehat{u}_{h}, \\
& \delta_{\Bq} = \Pi_W\Bq - \Bq, \quad \delta_{u} = \Pi_Vu - u, \quad \delta_{\widehat{u}} = \Pi_{\widehat{V}}u - u,
\end{align*}
where $\Pi_W$, $\Pi_V$ and $\Pi_{\widehat{V}}$ denote the standard $L^2$-orthogonal projections onto $\BW(K)$, $V(K)$ and $\mathcal{P}^k(F)$ on each
$K\in\mathcal{T}_h$ and $F\in\mathcal{E}_h$. Then
we obtain the following error equations with (\ref{HDG}) and the definitions of $\Pi_W$, $\Pi_V$ and $\Pi_{\widehat{V}}$:
\begin{subequations}
\label{error_eqs}
\begin{align}
\label{error_eq1}
 &(e_{\Bq}, \Bw_h)_{\mathcal{T}_{h}} - (e_{u}, \nabla\cdot \Bw_h)_{\mathcal{T}_{h}}
+ \langle e_{\widehat{u}}, \Bw_h\cdot\Bn \rangle_{\partial\mathcal{T}_{h}} = 0, \\
\label{error_eq2}
 &(\partial_{t}e_{u}, v_h)_{\mathcal{T}_{h}} - (e_{\Bq}, \nabla v_h)_{\mathcal{T}_{h}}
+ \langle e_{\Bq}\cdot\Bn + \tau (\mathcal{L}e_{u} - e_{\widehat{u}}), v_h - \widehat{v}_h \rangle_{\partial\mathcal{T}_{h}}
=  \langle \delta_{\Bq}\cdot\Bn + \tau\mathcal{L}\delta_{u}, v_h - \widehat{v}_h \rangle_{\partial\mathcal{T}_{h}},
\end{align}
\end{subequations}
for any $t\in[0,T]$ and any $(\Bw_h,v_h,\widehat{v}_h) \in \BW_{h} \times V_{h} \times \widehat{V}_{h}$.

\begin{lemma}\label{H1}
Let $e_{\Bq}$, $e_u$ and $e_{\widehat{u}}$ satisfy (\ref{error_eq1}), then we have
\begin{align}
\|\nabla e_u\|^2_{L^2(K)}+h_K^{-1}\|e_u-e_{\widehat{u}}\|^2_{L^2(\partial K)}\leq C\big(\|e_{\Bq}\|_{L^2(K)}^2+\|\sqrt{\tau}(\mathcal{L}e_u-e_{\widehat{u}})\|^2_{L^2(\partial K)}\big),
\end{align}
for each $K\in\mathcal{T}_h$.
\end{lemma}
\begin{proof}
For HDG-1 and HDG-2 methods, the proof can be found in \cite[Theorem 2.3]{cfq2018} and \cite[Lemma 3.2]{qs2016}. As for hybrid mixed DG methods, we know from
\cite[Lemma 3.1]{es2010} that there exists a unique solution $\tilde{\Bq}_h\in \BW_h$ such that
\begin{subequations}\label{H1-proof}
\begin{align}
(\tilde{\Bq}_h,\Bp_h)_K=&(\nabla v_h,\Bp_h)_K,\quad \forall \Bp_h\in (\mathcal{P}^{k-1}(K))^d,\\
\langle\tilde{\Bq}_h\cdot\Bn,w_h\rangle_{\partial K}=&\langle\mu_h,w_h\rangle_{\partial K},\quad \forall w_h\in\mathcal{P}^k(\partial K),
\end{align}
\end{subequations}
for each $(v_h,\mu_h)\in V_h\times \widehat{V}_h$ and $K\in\mathcal{T}_h$. Moreover, there exists a constant depending on the shape regularity of
the partition $\mathcal{T}_h$ such that
\begin{align}
\|\tilde{\Bq}_h\|_{L^2(K)}\leq C\big(\|\nabla v_h\|_{L^2(K)}+h_K^{\frac{1}{2}}\|\mu_h\|_{L^2(\partial K)}\big).\label{H1-proof:1}
\end{align}
By setting $(v_h,\mu_h)=(e_u, -h_K^{-1}(e_u-e_{\widehat{u}}))$ in (\ref{H1-proof}), we have by (\ref{error_eq1})
\begin{align*}
\|\nabla e_u\|^2_{L^2(K)}=&-(e_{\Bq},\tilde{\Bq}_h)_K+\langle e_u-e_{\widehat{u}},\tilde{\Bq}_h\cdot\Bn\rangle_{\partial K}\\
=&-(e_{\Bq},\tilde{\Bq}_h)_K-h_K^{-1}\| e_u-e_{\widehat{u}}\|^2_{L^2(\partial K)},
\end{align*}
which, together with (\ref{H1-proof:1}), leads to the desired result.
\end{proof}

\begin{lemma}\label{local_energy_estimate1}
Let $e_{\Bq}$, $e_u$, $e_{\widehat{u}}$, $\delta_{\Bq}$ and $\delta_u$ be defined as above, then if $u\in L^2((0,T);H^{1+\gamma}(\Omega))$ with $\partial_t u\in
L^2((0,T);L^2(\Omega))$ we have
\begin{align}
&\|e_{\Bq}\|_{L^2(I_0;L_h^2(\Omega_0))}+\|\nabla e_u\|_{L^2(I_0;L_h^2(\Omega_0))}+\Big(\int_{I_0}\sum_{K\in\mathcal{T}_h}\|\sqrt{\tau}(\mathcal{L}e_u-e_{\widehat{u}})\|^2_{L^2(\partial K\cap\Omega_0)} dt\Big)^{1/2}\nonumber\\
\leq & Ch\rho^{-1}\Big(\|e_{\Bq}\|_{L^2(I_1;L_h^2(\Omega_1))}
+\Big(\int_{I_1}\sum_{K\in\mathcal{T}_h(\Omega_2)}
\|\sqrt{\tau}(\mathcal{L}e_u-e_{\widehat{u}})\|^2_{L^2(\partial K)} dt\Big)^{1/2}\Big)\label{local_energy_estimate1:1}\\
&+C\|e_u(0)\|_{L_h^2(\Omega_1)}+C\rho^{-1}\|e_u\|_{L^2(I_1;L_h^2(\Omega_1))}+C\big(h^{-1}\|\delta_u\|_{L^2(I_1;L_h^2(\Omega_1))}\nonumber\\
&+\|\nabla\delta_u\|_{L^2(I_1;L_h^2(\Omega_1))}+\|\delta_{\Bq}\|_{L^2(I_1;L_h^2(\Omega_1))}
+h^{\gamma}\|\delta_{\Bq}\|_{L^2(I_1;H_h^{\gamma}(\Omega_1))}\big),\nonumber
\end{align}
where $\mathcal{T}_h(\Omega_2)=\{K\in\mathcal{T}_h: K\cap\Omega_2\neq \emptyset\}$.
\end{lemma}

\begin{proof}
By choosing $\Bw_h = (\varphi_{1,h} \varphi_{2})^{2}e_{\Bq}$, $v_h = (\varphi_{1,h} \varphi_{2})^{2}e_{u}$,
and $\widehat{v}_h = (\widehat{\varphi}_{1,h} \varphi_{2})^{2}e_{\widehat{u}}$ in the error equation (\ref{error_eqs}), we have
\begin{align*}
&(\partial_{t}e_{u}, (\varphi_{1,h} \varphi_{2})^{2}e_{u})_{\mathcal{T}_{h}}
- (e_{\Bq}, (\varphi_{1,h} \varphi_{2})^{2}\nabla e_{u})_{\mathcal{T}_{h}}
+ \langle e_{\Bq}\cdot\Bn + \tau (\mathcal{L}e_{u} - e_{\widehat{u}}),
(\varphi_{1,h} \varphi_{2})^{2}e_{u} - (\widehat{\varphi}_{1,h}
\varphi_{2})^{2}e_{\widehat{u}}\rangle_{\partial\mathcal{T}_{h}} \\
& \qquad + (\nabla e_{u}, (\varphi_{1,h} \varphi_{2})^{2}e_{\Bq})_{\mathcal{T}_{h}}
+ (e_{\Bq}, (\varphi_{1,h} \varphi_{2})^{2}e_{\Bq})_{\mathcal{T}_{h}}
- \langle e_{u} - e_{\widehat{u}}, (\varphi_{1,h} \varphi_{2})^{2}e_{\Bq}\cdot\Bn \rangle_{\partial\mathcal{T}_{h}} \\
&=  \langle \delta_{\Bq}\cdot\Bn + \tau\mathcal{L}\delta_{u},
(\varphi_{1,h} \varphi_{2})^{2}e_{u} - (\widehat{\varphi}_{1,h} \varphi_{2})^{2}
e_{\widehat{u}} \rangle_{\partial\mathcal{T}_{h}}.
\end{align*}
By a simple calculation and rearranging the terms in the above equation, we have
\begin{align}
& \frac{1}{2}\frac{d}{dt}\|(\varphi_{1,h} \varphi_{2})e_{u}\|^2_{L^2(\Omega)}
+ (e_{\Bq}, (\varphi_{1,h} \varphi_{2})^{2}e_{\Bq})_{\mathcal{T}_{h}}
+ \langle \tau (\mathcal{L}e_{u} - e_{\widehat{u}}),
(\varphi_{1,h} \varphi_{2})^{2}(\mathcal{L}e_{u} - e_{\widehat{u}})\rangle_{\partial\mathcal{T}_{h}} \nonumber\\
= & (\varphi_{1,h} e_u\partial_t\varphi_2,(\varphi_{1,h}\varphi_2)e_u)_{\mathcal{T}_h}+\langle \tau(\mathcal{L}e_{u} - e_{\widehat{u}}), ((\varphi_{1,h}\varphi_{2})^{2} - (\widehat{\varphi}_{1,h}\varphi_{2})^{2})
(e_{u} - e_{\widehat{u}})\rangle_{\partial\mathcal{T}_{h}}  \nonumber\\
& - \langle \tau(\mathcal{L}e_{u} - e_{\widehat{u}}), ((\varphi_{1,h}\varphi_{2})^{2} - (\widehat{\varphi}_{1,h}\varphi_{2})^{2})
e_{u} \rangle_{\partial\mathcal{T}_{h}} \label{local_energy_identity1}\\
& - \langle e_{\Bq}\cdot\Bn, ((\varphi_{1,h}\varphi_{2})^{2} - (\widehat{\varphi}_{1,h}\varphi_{2})^{2})
e_{u} \rangle_{\partial\mathcal{T}_{h}} \nonumber\\
& + \langle e_{\Bq}\cdot\Bn, ((\varphi_{1,h}\varphi_{2})^{2} - (\widehat{\varphi}_{1,h}\varphi_{2})^{2})
(e_{u} - e_{\widehat{u}})\rangle_{\partial\mathcal{T}_{h}} \nonumber\\
& + \langle \delta_{\Bq}\cdot\Bn + \tau\mathcal{L}\delta_{u},
(\varphi_{1,h} \varphi_{2})^{2}e_{u} - (\widehat{\varphi}_{1,h} \varphi_{2})^{2}
e_{\widehat{u}} \rangle_{\partial\mathcal{T}_{h}} =\sum_{i=1}^6 J_i.\nonumber
\end{align}
Next we are going to estimate $J_i~(i=1,2,\cdots,6)$ separately. By (\ref{supera}), Lemma \ref{inv}, Lemma \ref{tra} and Young's inequality, we deduce that
\begin{align}
J_1+J_2\leq& C\varphi_2\rho^{-2}\|\varphi_{1,h}e_u\|^2_{L^2(\Omega)}+\varphi^2_2\langle\tau(\mathcal{L}e_u-e_{\widehat{u}}),(\varphi_{1,h}-\widehat{\varphi}_{1,h})
(\varphi_{1,h}+\widehat{\varphi}_{1,h})(e_u-e_{\widehat{u}})\rangle_{\partial\mathcal{T}_h}\nonumber\\
\leq& C\varphi_2\rho^{-2}\|\varphi_{1,h}e_u\|^2_{L^2(\Omega)}+\frac{C}{\epsilon}\varphi_2^2(h\rho^{-1})^2\sum_{K\in\mathcal{T}_h(\Omega_2)}\|\sqrt{\tau}(e_u-e_{\widehat{u}})\|^2_{L^2(\partial K)}\label{local_energy_identity2}\\
&+\epsilon\sum_{K\in\mathcal{T}_h}\|\sqrt{\tau}(\varphi_{1,h}\varphi_2)(\mathcal{L}e_u-e_{\widehat{u}})\|^2_{L^2(\partial K)},\nonumber
\end{align}
and
\begin{align}
J_3\leq& C\varphi_2h\rho^{-1}\sum_{K\in\mathcal{T}_h}\|\sqrt{\tau}(\varphi_{1,h}\varphi_2)(\mathcal{L}e_u-e_{\widehat{u}})\|_{L^2(\partial K)}\|\sqrt{\tau}e_u\|_{L^2(\partial K)}\nonumber\\
\leq& C\varphi_2\rho^{-1}\sum_{K\in\mathcal{T}_h}\|\sqrt{\tau}(\varphi_{1,h}\varphi_2)(\mathcal{L}e_u-e_{\widehat{u}})\|_{L^2(\partial K)}\|e_u\|_{L^2(K)}\label{local_energy_identity3}\\
\leq& \epsilon\sum_{K\in\mathcal{T}_h}\|\sqrt{\tau}(\varphi_{1,h}\varphi_2)(\mathcal{L}e_u-e_{\widehat{u}})\|_{L^2(\partial K)}^2+\frac{C}{\epsilon}\varphi_2^2\rho^{-2}
\|e_u\|^2_{L^2_h(\Omega_1)},\nonumber
\end{align}
where $\epsilon>0$ (independent of $h$ and $\rho$) is a constant that can be arbitrary small. Similarly, with (\ref{supera}), Lemma \ref{inv}, Lemma \ref{tra}, Lemma \ref{H1} and Young's inequality, we also can obtain that
\begin{align}
J_4+J_5\leq& C\varphi_2h\rho^{-1}\sum_{K\in\mathcal{T}_h}\|(\varphi_{1,h}\varphi_2)e_{\Bq}\|_{L^2(\partial K)}\big(\|e_u\|_{L^2(\partial K)}
+\|e_u-e_{\widehat{u}}\|_{L^2(\partial K)}\big)\nonumber\\
\leq&\epsilon\|(\varphi_{1,h}\varphi_2)e_{\Bq}\|_{L^2(\Omega)}^2+\frac{C}{\epsilon}\varphi_2^2\rho^{-2}\|e_u\|^2_{L_h^2(\Omega_1)}+\frac{C}{\epsilon}
\varphi_2^2(h\rho^{-1})^2\sum_{K\in\mathcal{T}_h(\Omega_2)}h_K^{-1}\|e_u-e_{\widehat{u}}\|_{L^2(\partial K)}^2\label{local_energy_identity4}\\
\leq&\epsilon\|(\varphi_{1,h}\varphi_2)e_{\Bq}\|_{L^2(\Omega)}^2+\frac{C}{\epsilon}\varphi_2^2\rho^{-2}\|e_u\|^2_{L_h^2(\Omega_1)}\nonumber\\
&+\frac{C}{\epsilon}
\varphi_2^2(h\rho^{-1})^2\Big(\|e_{\Bq}\|^2_{L^2_h(\Omega_1)}+\sum_{K\in\mathcal{T}_h(\Omega_2)}\|\sqrt{\tau}(\mathcal{L}e_u-e_{\widehat{u}})\|^2_{L^2(\partial K)}\Big),\nonumber
\end{align}
and
\begin{align}
J_6=&\langle\delta_{\Bq}\cdot\textbf{n}+\tau\mathcal{L}\delta_u,\varphi_2^2(\varphi_{1,h}^2-\widehat{\varphi}_{1,h}^2)e_u\rangle_{\partial\mathcal{T}_h}
+\langle\delta_{\Bq}\cdot\textbf{n}+\tau\mathcal{L}\delta_u,(\widehat{\varphi}_{1,h}\varphi_2)^2(e_u-e_{\widehat{u}})\rangle_{\partial\mathcal{T}_h}\nonumber\\
\leq&C\varphi_2\sum_{K\in\mathcal{T}_h}\Big(\|\delta_{\Bq}\|_{L^2(\partial K)}+\|\tau\delta_u\|_{L^2(\partial K)}\Big)\Big(h\rho^{-1}\|\varphi_{1,h}e_u\|_{L^2(\partial K)}\nonumber\\
&+\|(\varphi_{1,h}\varphi_2)(e_u-e_{\widehat{u}})\|_{L^2(\partial K)}\Big)\nonumber\\
\leq&C\varphi_2\sum_{K\in\mathcal{T}_h}\Big(h_K^{-1}\|\delta_u\|_{L^2(K)}+\|\nabla\delta_u\|_{L^2(K)}+\|\delta_{\Bq}\|_{L^2(K)}
+h_K^{\gamma}|\delta_{\Bq}|_{H^{\gamma}(K)}\Big)\times\label{local_energy_identity5}\\
&\Big(\rho^{-1}\|\varphi_{1,h}e_u\|_{L^2(K)}+h_K^{-\frac{1}{2}}\|(\varphi_{1,h}\varphi_2)(e_u-e_{\widehat{u}})\|_{L^2(\partial K)}\Big)\nonumber\\
\leq&C\varphi_2\rho^{-2}\|e_u\|^2_{L_h^2(\Omega_1)}+C\varphi_2\left(h^{-2}\|\delta_u\|^2_{L_h^2(\Omega_1)}+\|\nabla\delta_u\|^2_{L_h^2(\Omega_1)}
+\|\delta_{\Bq}\|^2_{L_h^2(\Omega_1)}+h^{2\gamma}|\delta_{\Bq}|^2_{H_h^{\gamma}(\Omega_1)}\right)\nonumber\\
&+C\epsilon\left(\|(\varphi_{1,h}\varphi_2)e_{\Bq}\|^2_{L^2(\Omega)}+\sum_{K\in\mathcal{T}_h}
\|\sqrt{\tau}(\varphi_{1,h}\varphi_2)(\mathcal{L}e_u-e_{\widehat{u}})\|^2_{L^2(\partial K)}\right).\nonumber
\end{align}
Then collecting (\ref{local_energy_identity1})-(\ref{local_energy_identity5}), we arrive at
\begin{align}
&\frac{d}{dt}\|(\varphi_{1,h} \varphi_{2})e_{u}\|^2_{L^2(\Omega)}+\|(\varphi_{1,h}\varphi_2)e_{\Bq}\|^2_{L^2(\Omega)}+\sum_{K\in\mathcal{T}_h}
\|\sqrt{\tau}(\varphi_{1,h}\varphi_2)(\mathcal{L}e_u-e_{\widehat{u}})\|^2_{L^2(\partial K)}\nonumber\\
\leq& C\varphi_2\rho^{-2}\|e_u\|^2_{L_h^2(\Omega_1)}+C\varphi_2^2(h\rho^{-1})^2\Big(\|e_{\Bq}\|^2_{L_h^2(\Omega_1)}+\sum_{K\in\mathcal{T}_h(\Omega_2)}
\|\sqrt{\tau}(\mathcal{L}e_u-e_{\widehat{u}})\|^2_{L^2(\partial K)}\Big)\label{local_energy_identity6}\\
&+C\varphi_2\left(h^{-2}\|\delta_u\|^2
_{L_h^2(\Omega_1)}+\|\nabla\delta_u\|^2_{L_h^2(\Omega_1)}+\|\delta_{\Bq}\|^2_{L_h^2(\Omega_1)}+h^{2\gamma}|\delta_{\Bq}|^2_{H_h^{\gamma}(\Omega_1)}\right),\nonumber
\end{align}
by setting $\epsilon$ to be sufficiently small.

If $0$ is not the left endpoint of $I_1$, then it follows from the definition of $\varphi_2$ that $\varphi_2=0$ at that endpoint. Hence,
\begin{align*}
\int_{I_1}\frac{d}{dt}\|(\varphi_{1,h} \varphi_{2})e_{u}\|^2_{L^2(\Omega)} dt=\|(\varphi_{1,h}\varphi_2(t^{\ast}))e_u(t^{\ast})\|^2_{L^2(\Omega)}\geq 0,
\end{align*}
where $t^{\ast}$ denotes the right endpoint of $I_1$.
If $0$ is the left endpoint of $I_1$, then by the definition of $\varphi_2$ we have $0\leq \varphi_2(0)\leq 1$ and
\begin{align*}
\int_{I_1}\frac{d}{dt}\|(\varphi_{1,h} \varphi_{2})e_{u}\|^2_{L^2(\Omega)} dt=&\|(\varphi_{1,h}\varphi_2(t^{\ast}))e_u(t^{\ast})\|^2_{L^2(\Omega)}-
\|(\varphi_{1,h}\varphi_2(0))e_u(0)\|^2_{L^2(\Omega)}\\
\geq& -\|\varphi_{1,h}e_u(0)\|^2_{L^2(\Omega)}.
\end{align*}
Therefore, applying Lemma \ref{H1} and the above two inequalities and integrating (\ref{local_energy_identity6}) over $I_1$ in time we obtain
\begin{align*}
&\|e_{\Bq}\|_{L^2(I_0;L_h^2(\Omega_0))}+\|\nabla e_u\|_{L^2(I_0;L_h^2(\Omega_0))}+\Big(\int_{I_0}\sum_{K\in\mathcal{T}_h}\|\sqrt{\tau}(\mathcal{L}e_u-e_{\widehat{u}})\|^2_{L^2(\partial K\cap\Omega_0)} dt\Big)^{1/2}\\
\leq & Ch\rho^{-1}\Big(\|e_{\Bq}\|_{L^2(I_1;L_h^2(\Omega_1))}
+\Big(\int_{I_1}\sum_{K\in\mathcal{T}_h(\Omega_2)}
\|\sqrt{\tau}(\mathcal{L}e_u-e_{\widehat{u}})\|^2_{L^2(\partial K)} dt\Big)^{1/2}\Big)\\
&+C\|e_u(0)\|_{L_h^2(\Omega_1)}+C\rho^{-1}\|e_u\|_{L^2(I_1;L_h^2(\Omega_1))}+C\big(h^{-1}\|\delta_u\|_{L^2(I_1;L_h^2(\Omega_1))}\\
&+\|\nabla\delta_u\|_{L^2(I_1;L_h^2(\Omega_1))}+\|\delta_{\Bq}\|_{L^2(I_1;L_h^2(\Omega_1))}
+h^{\gamma}\|\delta_{\Bq}\|_{L^2(I_1;H_h^{\gamma}(\Omega_1))}\big).
\end{align*}
\end{proof}

\begin{lemma}\label{local_energy_estimate2}
Let $e_{\Bq}$, $e_u$, $e_{\widehat{u}}$, $\delta_{\Bq}$ and $\delta_u$ be defined as above, then if $u\in L^{\infty}((0,T);H^{1+\gamma}(\Omega))$ with $\partial_t u\in L^{2}((0,T);H^{1+\gamma}(\Omega))$ we have
\begin{align}
\|\partial_t e_u\|_{L^2(I_0;L_h^2(\Omega_0))}\leq & C\rho^{-1}\Big(\|e_{\Bq}\|_{L^2(I_1;L_h^2(\Omega_1))}+\Big(\int_{I_1}\sum_{K\in\mathcal{T}_h(\Omega_2)}\|\sqrt{\tau}(\mathcal{L}e_u-e_{\widehat{u}})\|^2_{L^2(\partial K)}dt\Big)^{1/2}\Big)\nonumber\\
&+Ch\Big(\|\partial_t e_{\Bq}\|_{L^2(I_1;L_h^2(\Omega_1))}+\Big(\int_{I_1}\sum_{K\in\mathcal{T}_h(\Omega_2)}\|\sqrt{\tau}\partial_t(\mathcal{L}e_u-e_{\widehat{u}})\|^2_{L^2(\partial K)}dt\Big)^{1/2}\Big)\nonumber\\
&+C\Big(h^{-1}\rho^{-1}\|\delta_u\|_{L^2(I_1;L_h^2(\Omega_1))}+\rho^{-1}\|\nabla \delta_u\|_{L^2(I_1;L_h^2(\Omega_1))}+h^{-1}\rho\|\partial_t\delta_u\|_{L^2(I_1;L_h^2(\Omega_1))}\label{local_energy_estimates2:1}\\
&+\rho\|\nabla\partial_t\delta_u\|_{L^2(I_1;L_h^2(\Omega_1))}+\rho^{-1}\|\delta_{\Bq}\|_{L^2(I_1;L_h^2(\Omega_1))}+h^{\gamma}\rho^{-1}
\|\delta_{\Bq}\|_{L^2(I_1;H_h^{\gamma}(\Omega_1))}
\nonumber\\
&+\rho\|\partial_t\delta_{\Bq}\|_{L^2(I_1;L_h^2(\Omega_1))}+h^{\gamma}\rho\|\partial_t\delta_{\Bq}\|_{L^2(I_1;H_h^{\gamma}(\Omega_1))}\Big)+C\Big(\|\delta_{\Bq}\|_{L^{\infty}(I_1;L_h^2(\Omega_1))}\nonumber\\
&+h^{\gamma}\|\delta_{\Bq}\|_{L^{\infty}(I_1;H_h^{\gamma}(\Omega_1))}+h^{-1}\|\delta_{u}\|_{L^{\infty}(I_1;L_h^2(\Omega_1))}+\|\nabla\delta_{u}\|_{L^{\infty}(I_1;L_h^2(\Omega_1))}\Big)\nonumber\\
&+C\Big(\|e_{\Bq}(0)\|_{L_h^2(\Omega_1)}+\Big(\sum_{K\in\mathcal{T}_h(\Omega_2)}\|\sqrt{\tau}(\mathcal{L}e_u-e_{\widehat{u}})(0)\|^2_{L^2(\partial K)}\Big)^{1/2}\Big).\nonumber
\end{align}
\end{lemma}
\begin{proof}
By calculating the derivative with respect to the time from both sides of (\ref{error_eq1}), we derive that
\begin{align*}
(\partial_t e_{\Bq},\Bw_h)_{\mathcal{T}_h}+(\nabla \partial_t e_u,\Bw_h)_{\mathcal{T}_h}-\langle\partial_t(e_u-e_{\widehat{u}}),\Bw_h\cdot\Bn\rangle
_{\partial\mathcal{T}_h}=0,\quad \forall \Bw_h\in \BW_h.
\end{align*}
Then we set $\Bw_h=(\varphi_{1,h}\varphi_2)^2 e_{\Bq}$ in the above equation and $v_h=(\varphi_{1,h}\varphi_2)^2\partial_t e_u$ and
$\widehat{v}_h=(\widehat{\varphi}_{1,h}\varphi_2)^2\partial_t e_{\widehat{u}}$ in (\ref{error_eq2}) to infer that
\begin{align}
&\|(\varphi_{1,h}\varphi_2)\partial_t e_u\|^2_{L^2(\Omega)}+\frac{1}{2}\frac{d}{dt}\|(\varphi_{1,h}\varphi_2)e_{\Bq}\|^2_{L^2(\Omega)}+
\frac{1}{2}\frac{d}{dt}\sum_{K\in\mathcal{T}_h}\|\sqrt{\tau}(\varphi_{1,h}\varphi_2)(\mathcal{L}e_u-e_{\widehat{u}})\|_{L^2(\partial K)}^2\nonumber\\
=&-\langle\tau(\mathcal{L}e_u-e_{\widehat{u}}),(\varphi_{1,h}^2-\widehat{\varphi}_{1,h}^2)\varphi_2^2\partial_t e_{\widehat{u}}\rangle_{\partial\mathcal{T}_h}
-\langle e_{\Bq}\cdot\Bn,(\varphi_{1,h}^2-\widehat{\varphi}_{1,h}^2)\varphi_2^2\partial_t e_{\widehat{u}}\rangle_{\partial\mathcal{T}_h}\label{local_energy_estimate2-proof:1}\\
&+\langle\delta_{\Bq}\cdot\Bn+\tau\mathcal{L}\delta_u, (\varphi_{1,h}\varphi_2)^2\partial_t e_u-(\widehat{\varphi}_{1,h}\varphi_2)^2\partial_t e_{\widehat{u}}\rangle
_{\partial\mathcal{T}_h}\nonumber\\
&+\varphi_2\partial_t\varphi_2\|\varphi_{1,h}e_{\Bq}\|^2_{L^2(\Omega)}+\varphi_2\langle\tau(\varphi_{1,h}(\mathcal{L}e_u-e_{\widehat{u}}))^2,
\partial_t\varphi_2\rangle_{\partial\mathcal{T}_h}=\sum_{i=7}^{11}J_i.\nonumber
\end{align}

Next we are going to estimate $J_i~(i=7,\cdots,11)$ separately. By the definition of $\varphi_2$, we have
\begin{align}
J_{10}+J_{11}\leq C\varphi_2\rho^{-2}\left(\|\varphi_{1,h}e_{\Bq}\|^2_{L^2(\Omega)}+\sum_{K\in\mathcal{T}_h}\|\sqrt{\tau}\varphi_{1,h}(\mathcal{L}e_u-e_{\widehat{u}})\|^2_{L^2(\partial K)}\right).\label{local_energy_estimate2-proof:2}
\end{align}
By (\ref{supera}), Lemma \ref{inv}, Lemma \ref{tra}, Lemma \ref{H1} and Young's inequality, we have
\begin{align}
J_{9}=& \langle\delta_{\Bq}\cdot\textbf{n}+\tau\mathcal{L}\delta_u, (\varphi_{1,h}^2-\widehat{\varphi}_{1,h}^2)\varphi_2^2\partial_t e_u\rangle_{\partial\mathcal{T}_h}
+\langle\delta_{\Bq}\cdot\textbf{n}+\tau\mathcal{L}\delta_u, (\widehat{\varphi}_{1,h}\varphi_2)^2\partial_t (e_u-e_{\widehat{u}})\rangle_{\partial\mathcal{T}_h}\nonumber\\
\leq & C\varphi_2h\rho^{-1}\sum_{K\in\mathcal{T}_h}\big(h_K^{-\frac{1}{2}}\|\delta_{\Bq}\|_{L^2(K)}+h_K^{\gamma-\frac{1}{2}}|\delta_{\Bq}|_{H^{\gamma}(K)}
+\|\tau\delta_u\|_{L^2(\partial K)}\big)\|(\varphi_{1,h}\varphi_2)\partial_t e_u\|_{L^2(\partial K)}\nonumber\\
&+\frac{\partial}{\partial t}\langle\delta_{\Bq}\cdot\textbf{n}+\tau\mathcal{L}\delta_u,(\widehat{\varphi}_{1,h}\varphi_2)^2(e_u-e_{\widehat{u}})\rangle_{\partial\mathcal{T}_h}-
\langle\partial_t\delta_{\Bq}\cdot\textbf{n}+\tau\mathcal{L}(\partial_t\delta_u),(\widehat{\varphi}_{1,h}\varphi_2)^2(e_u-e_{\widehat{u}})\rangle_{\partial\mathcal{T}_h}\nonumber\\
&-2\varphi_2\partial_t\varphi_2\langle\delta_{\Bq}\cdot\textbf{n}+\tau\mathcal{L}\delta_u,\widehat{\varphi}_{1,h}^2(e_u-e_{\widehat{u}})\rangle_{\partial\mathcal{T}_h}
\nonumber\\
\leq &\epsilon\|(\varphi_{1,h}\varphi_2)\partial_t e_u\|^2_{L^2(\Omega)}+C\varphi_2\big(\rho^{-2}\|\delta_{\Bq}\|_{L_h^2(\Omega_1)}^2+h^{2\gamma}\rho^{-2}
|\delta_{\Bq}|_{H_h^{\gamma}(\Omega_1)}^2+h^{-2}\rho^{-2}\|\delta_u\|_{L_h^2(\Omega_1)}^2\\
&+\rho^{-2}\|\nabla\delta_u\|^2_{L_h^2(\Omega_1)}+\rho^2\|\partial_t\delta_{\Bq}\|_{L_h^2(\Omega_1)}^2+h^{2\gamma}\rho^2
|\partial_t\delta_{\Bq}|^2_{H_h^{\gamma}(\Omega_1)}
+h^{-2}\rho^2\|\partial_t\delta_u\|^2_{L_h^2(\Omega_1)}\nonumber\\
&+\rho^2\|\nabla\partial_t\delta_u\|^2_{L_h^2(\Omega_1)}\big)+C\varphi_2\rho^{-2}\Big(\|e_{\Bq}\|^2_{L^2_h(\Omega_1)}+\sum_{K\in\mathcal{T}_h(\Omega_2)}
\|\sqrt{\tau}(\mathcal{L}e_u-e_{\widehat{u}})\|^2_{L^2(\partial K)}\Big)\nonumber\\
&+\frac{\partial}{\partial t}\langle\delta_{\Bq}\cdot\textbf{n}+\tau\mathcal{L}\delta_u,(\widehat{\varphi}_{1,h}\varphi_2)^2(e_u-e_{\widehat{u}})\rangle_{\partial\mathcal{T}_h},\nonumber
\end{align}
where $\epsilon$ (independent of $h$ and $\rho$) is a positive constant that can be arbitrary small.
Now we turn to estimate $J_7+J_8$. Similarly, by using (\ref{supera}), Lemma \ref{inv}, Lemma \ref{tra}, Lemma \ref{H1} and  Young's inequality, we can obtain that
\begin{align}
J_7+J_8=&\langle e_{\Bq}\cdot\textbf{n}+\tau(\mathcal{L}e_u-e_{\widehat{u}}),(\varphi_{1,h}^2-\widehat{\varphi}_{1,h}^2)\varphi_2^2\partial_t(e_u-e_{\widehat{u}})\rangle_{\partial\mathcal{T}_h}\nonumber\\
&-\langle e_{\Bq}\cdot\textbf{n}+\tau(\mathcal{L}e_u-e_{\widehat{u}}),(\varphi_{1,h}^2-\widehat{\varphi}_{1,h}^2)\varphi_2^2\partial_te_u\rangle_{\partial\mathcal{T}_h}\nonumber\\
\leq &\epsilon\|(\varphi_{1,h}\varphi_2)\partial_t e_u\|^2_{L^2(\Omega)}+C\varphi_2^2\rho^{-2}\Big(\|e_{\Bq}\|^2_{L_h^2(\Omega_1)}+\sum_{K\in\mathcal{T}_h(\Omega_2)}\|\sqrt{\tau}(\mathcal{L}e_u-e_{\widehat{u}})\|^2_{L^2(\partial K)}\Big)\label{local_energy_estimate2-proof:4}\\
&+C\varphi_2^2h^2\sum_{K\in\mathcal{T}_h(\Omega_2)}h_K^{-1}\|\partial_t(e_u-e_{\widehat{u}})\|^2_{L^2(\partial K)} \nonumber\\
\leq &\epsilon\|(\varphi_{1,h}\varphi_2)\partial_t e_u\|^2_{L^2(\Omega)}+C\varphi_2^2\rho^{-2}\Big(\|e_{\Bq}\|^2_{L_h^2(\Omega_1)}+\sum_{K\in\mathcal{T}_h(\Omega_2)}\|\sqrt{\tau}(\mathcal{L}e_u-e_{\widehat{u}})\|^2_{L^2(\partial K)}\Big)\nonumber\\
&+C\varphi_2^2h^2\Big(\|\partial_te_{\Bq}\|_{L_h^2(\Omega_1)}^2+\sum_{K\in\mathcal{T}_h(\Omega_2)}\|\sqrt{\tau}\partial_t(\mathcal{L}e_u-e_{\widehat{u}})\|^2_{L^2(\partial K)}\Big).\nonumber
\end{align}
Finally by applying (\ref{local_energy_estimate2-proof:2})-(\ref{local_energy_estimate2-proof:4}) to (\ref{local_energy_estimate2-proof:1}) and
doing time integration over $I_1$, we obtain the desired result.
\end{proof}
\begin{lemma}[local energy error estimate]\label{local_energy_estimate3}
Let $e_{u}$, $e_{\widehat{u}}$, $e_{\Bq}$, $\delta_u$, $\delta_{\Bq}$ be defined as above, then if $u\in L^{\infty}((0,T);H^{1+\gamma}(\Omega))$ with
$\partial_t u\in L^2((0,T);H^{1+\gamma}(\Omega))$ we have
\begin{align}
&\|\partial_t(u-u_h)\|_{L^2(I_0;L_h^2(\Omega_0))}+\rho^{-1}\Big(\|\nabla(u-u_h)\|_{L^2(I_0;L_h^2(\Omega_0))}+\mathcal{K}_0\Big)\nonumber\\
\leq&C(h\rho^{-1})^{\beta_2}\rho^{-1}\mathcal{K}+C\mathcal{F}+C\Big(\rho^{-1}\|e_u(0)\|_{L^2_h(\Omega_1)}
+h\|\partial_te_u(0)\|_{L_h^2(\Omega_1)}\Big)\nonumber\\
&+C\rho^{-2}\|u-u_h\|_{L^2(I_1;L_h^2(\Omega_1))}
+Ch(h\rho^{-1})^{\beta_1}\mathcal{J},\nonumber
\end{align}
for any real numbers $\beta_1,\beta_2>0$, where $C$ is a positive constant depending on $\beta_1$ and $\beta_2$,
\begin{align*}
\mathcal{J}^2=&\|\partial_t e_{\Bq}\|^2_{L^2(I_1;L_h^2(\Omega_1))}+\int_{I_1}\sum_{K\in\mathcal{T}_h(\Omega_2)}\|\sqrt{\tau}\partial_t(\mathcal{L}e_u-e_{\widehat{u}})\|^2_{L^2(\partial K)}dt,\\
\mathcal{K}^2_{0}=&\|\Bq-\Bq_h\|^2_{L^2(I_0;L_h^2(\Omega_0))}+\int_{I_0}\sum_{K\in\mathcal{T}_h}\|\sqrt{\tau}(\mathcal{L}u_h-\widehat{u}_h)\|^2_{L^2(\partial K\cap\Omega_0)} dt,\\
\mathcal{K}^2=&\|\Bq-\Bq_h\|^2_{L^2(I_1;L_h^2(\Omega_1))}+\int_{I_1}\sum_{K\in\mathcal{T}_h(\Omega_2)}\|\sqrt{\tau}(\mathcal{L}u_h-\widehat{u}_h)\|^2_{L^2(\partial K)} dt,\\
\mathcal{F}=&h^{-1}\rho^{-1}\|\delta_u\|_{L^2(I_1;L_h^2(\Omega_1))}+\rho^{-1}\|\nabla \delta_u\|_{L^2(I_1;L_h^2(\Omega_1))}+h^{-1}\rho\|\partial_t\delta_u\|_{L^2(I_1;L_h^2(\Omega_1))}\\
&+\rho\|\nabla\partial_t\delta_u\|_{L^2(I_1;L_h^2(\Omega_1))}+\rho^{-1}\|\delta_{\Bq}\|_{L^2(I_1;L_h^2(\Omega_1))}
+h^{\gamma}\rho^{-1}\|\delta_{\Bq}\|_{L^2(I_1;H_h^{\gamma}(\Omega_1))}\\
&+\rho\|\partial_t\delta_{\Bq}\|_{L^2(I_1;L_h^2(\Omega_1))}+h^{\gamma}\rho\|\partial_t\delta_{\Bq}\|
_{L^2(I_1;H_h^{\gamma}(\Omega_1))}+\|\delta_{\Bq}\|_{L^{\infty}(I_1;L_h^2(\Omega_1))}\\
&+h^{\gamma}\|\delta_{\Bq}\|_{L^{\infty}(I_1;H_h^{\gamma}(\Omega_1))}+h^{-1}\|\delta_{u}\|_{L^{\infty}(I_1;L_h^2(\Omega_1))}
+\|\nabla\delta_{u}\|_{L^{\infty}(I_1;L_h^2(\Omega_1))}\\
&+\|e_{\Bq}(0)\|_{L_h^2(\Omega_1)}+\Big(\sum_{K\in\mathcal{T}_h(\Omega_2)}\|\sqrt{\tau}(\mathcal{L}e_u-e_{\widehat{u}})(0)\|^2_{L^2(\partial K)}\Big)^{1/2}.
\end{align*}
\end{lemma}
\begin{proof}
Let $\beta_1$ and $\beta_2$ be two positive real numbers, and $\lceil \beta_1\rceil$ and $\lceil\beta_2\rceil$ denote the smallest integers greater than or equal to $\beta_1$ and $\beta_2$ respectively. Then, we divide both the space domain between $\Omega_0$ and $\Omega_2$ and the time domain between $I_0$ and $I_1$ into $n+m$ ($n=3\lceil\beta_1\rceil+1$ and $m=\lceil\beta_2\rceil$) equivalent sub-domains
such that $\Omega_0\subset\Omega_{3}\subset\cdots\subset\Omega_{n+1}\subset\cdots\subset\Omega_{n+m+1}\subset\Omega_2$ and $I_0\subset I_{3}\subset\cdots\subset I_{n+1}\subset\cdots\subset I_{n+m+1}\subset I_1$ with ${\rm dist}(\partial\Omega_0, \partial\Omega_3\backslash\partial\Omega)={\rm dist}(\partial\Omega_{i}, \partial\Omega_{i+1}\backslash\partial\Omega)={\rm dist}(\partial\Omega_{n+m+1},\partial\Omega_2\backslash\partial\Omega)=\rho/(n+m)$ and ${\rm dist}(\partial I_{0}, \partial I_{3}\backslash\{0,T\})={\rm dist}(\partial I_{i}, \partial I_{i+1}\backslash\{0,T\})=
{\rm dist}(\partial I_{n+m+1}, \partial I_{1}\backslash\{0,T\})=2\rho^2/(n+m)$ for $i=3,\cdots,n+m$.

In the right-hand side of the inequality (\ref{local_energy_estimates2:1}), let
\begin{align*}
\mathcal{X}^2_i=&\|e_{\Bq}\|^2_{L^2(I_{i};L_h^2(\Omega_{i}))}+\int_{I_{i}}\sum_{K\in\mathcal{T}_h(\Omega_{i})}
\|\sqrt{\tau}(\mathcal{L}e_u-e_{\widehat{u}})\|^2_{L^2(\partial K)} dt,\\
\mathcal{J}^2_{i}=&\|\partial_t e_{\Bq}\|^2_{L^2(I_{i};L_h^2(\Omega_{i}))}+\int_{I_{i}}\sum_{K\in\mathcal{T}_h(\Omega_{i})}
\|\sqrt{\tau}\partial_t(\mathcal{L}e_u-e_{\widehat{u}})\|^2_{L^2(\partial K)} dt,
\end{align*}
then repeatedly applying the results obtained in Lemma \ref{local_energy_estimate1} and Lemma \ref{local_energy_estimate2} we have
\begin{align}
&\|\partial_t e_u\|_{L^2(I_0;L^2(\Omega_0))}\nonumber\\
\leq& Ch\mathcal{J}_{3}+C\rho^{-1}\mathcal{X}_3+C\mathcal{F}\quad\quad ({\rm Lemma}~\ref{local_energy_estimate2})\nonumber\\
\leq& Ch(h\rho^{-1})\mathcal{J}_{4}+Ch\rho^{-1}\|\partial_t e_u\|_{L^2(I_{4};L_h^2(\Omega_{4}))}\nonumber\\
&+Ch\|\partial_t e_u(0)\|_{L_h^2(\Omega_1)}+C\rho^{-1}\mathcal{X}_3+C\mathcal{F}\quad({\rm Lemma}~\ref{local_energy_estimate1})\nonumber\\
\leq& Ch(h\rho^{-1})\mathcal{J}_{5}+Ch\|\partial_t e_u(0)\|_{L_h^2(\Omega_1)}+C\rho^{-1}\mathcal{X}_5+C\mathcal{F}\quad\quad ({\rm Lemma}~\ref{local_energy_estimate2})
\label{local_energy_estimate3-proof:1}\\
&\quad\cdots\cdots\nonumber\\
\leq& Ch(h\rho^{-1})^{\lceil\beta_1\rceil}\mathcal{J}_{n}+Ch\|\partial_t e_u(0)\|_{L_h^2(\Omega_1)}+C\rho^{-1}\mathcal{X}_{n}+C\mathcal{F}\nonumber\\
\leq& Ch(h\rho^{-1})^{\beta_1}\mathcal{J}_{n}+Ch\|\partial_t e_u(0)\|_{L_h^2(\Omega_1)}+C\rho^{-1}\mathcal{X}_{n}+C\mathcal{F},\nonumber
\end{align}
where $C$ is a positive constant depending on $n+m$. Note that in the above inequalities, we have used the fact that $h\rho^{-1}<1$.

Similarly, by repeatedly using Lemma \ref{local_energy_estimate1} we also have
\begin{align}
&\rho^{-1}\Big(\|\nabla e_u\|_{L^2(I_{n};L_h^2(\Omega_{n}))}+\mathcal{X}_{n}\Big)\nonumber\\
\leq& C(h\rho^{-1})\rho^{-1}\mathcal{X}_{n+1}+C\rho^{-1}\|e_u(0)\|_{L_h^2(\Omega_1)}+C\rho^{-2}\|e_u\|_{L^2(I_1;L_h^2(\Omega_1))}+C\mathcal{F}\nonumber\\
\leq& C(h\rho^{-1})^2\rho^{-1}\mathcal{X}_{n+2}+C\rho^{-1}\|e_u(0)\|_{L_h^2(\Omega_1)}+C\rho^{-2}\|e_u\|_{L^2(I_1;L_h^2(\Omega_1))}+C\mathcal{F}
\label{local_energy_estimate3-proof:2}\\
&\cdots\cdots\nonumber\\
\leq& C(h\rho^{-1})^{m}\rho^{-1}\mathcal{X}_{n+m}+C\rho^{-1}\|e_u(0)\|_{L_h^2(\Omega_1)}+C\rho^{-2}\|e_u\|_{L^2(I_1;L_h^2(\Omega_1))}+C\mathcal{F}\nonumber\\
\leq& C(h\rho^{-1})^{\beta_2}\rho^{-1}\mathcal{X}_{n+m}+C\rho^{-1}\|e_u(0)\|_{L_h^2(\Omega_1)}+C\rho^{-2}\|e_u\|_{L^2(I_1;L_h^2(\Omega_1))}+C\mathcal{F},\nonumber
\end{align}
which, together with (\ref{local_energy_estimate3-proof:1}) and the triangle inequality, leads to
\begin{align*}
&\|\partial_t(u-u_h)\|_{L^2(I_0;L_h^2(\Omega_0))}+\rho^{-1}\Big(\|\nabla(u-u_h)\|_{L^2(I_0;L_h^2(\Omega_0))}+\mathcal{K}_0\Big)\\
\leq &C\|\partial_t e_u\|_{L^2(I_0;L_h^2(\Omega_0))}+C\rho^{-1}\Big(\|\nabla e_u\|_{L^2(I_0;L_h^2(\Omega_0))}+\mathcal{X}_0\Big)
+C\|\partial_t\delta_u\|_{L^2(I_0;L_h^2(\Omega_0))}\\
&+C\rho^{-1}\big(\|\delta_{\Bq}\|_{L^2(I_0;L_h^2(\Omega_0))}
+\|\nabla\delta_u\|_{L^2(I_0;L_h^2(\Omega_{2}))}+h^{-1}\|\delta_u\|_{L^2(I_0;L_h^2(\Omega_{2}))}\big)\\
\leq&C\rho^{-1}\Big(\|\nabla e_u\|_{L^2(I_{n};L_h^2(\Omega_{n}))}+\mathcal{X}_{n}\Big)+Ch(h\rho^{-1})^{\beta_1}\mathcal{J}_{n}+Ch\|\partial_t e_u(0)\|_{L_h^2(\Omega_1)}
+C\mathcal{F}\\
\leq& C(h\rho^{-1})^{\beta_2}\rho^{-1}\mathcal{X}_{n+m}+Ch(h\rho^{-1})^{\beta_1}\mathcal{J}_{n}+C\Big(h\|\partial_t e_u(0)\|_{L_h^2(\Omega_1)}+\rho^{-1}\|e_u(0)\|_{L_h^2(\Omega_1)}\Big)\\
&+C\rho^{-2}\|u-u_h\|_{L^2(I_1;L_h^2(\Omega_1))}+C\mathcal{F}.
\end{align*}
We conclude the proof.
\end{proof}

\begin{remark}
In Lemma \ref{local_energy_estimate2} and Lemma \ref{local_energy_estimate3}, we required the regularity assumptions $u\in L^{\infty}((0,T);H^{1+\gamma}(\Omega))$ and $\partial_t u\in L^2((0,T);H^{1+\gamma}(\Omega))$. These assumptions are not satisfied by the standard parabolic equations in nonconvex domains. Nevertheless, in the subsequent proofs, these lemmas are only applied to the smooth regularized Green's function, which does satisfy these assumptions (see the text below (\ref{rGreenHDG})).
\end{remark}

\section{Maximum norm stability}\label{sec4}
In this section, we are devoted to proving Theorem~\ref{max-stability} - the maximum-norm stability, with the local energy error estimate established in Lemma \ref{local_energy_estimate3}.

To this end, for any $z\in \overline{K}_z$ with $K_z\in\mathcal{T}_h$, let $\delta_z$ denote the Dirac Delta function concentrated at $z$ such that
$\int_{\Omega}\delta_z(x)v(x) dx=v(z)$ for any $v(x)\in \mathcal{C}(\overline{\Omega})$, and $\widetilde{\delta}_z$ a smooth function supported
in $K_z$ such that (see \cite[Appendix A]{sw1995} and \cite{stw1998})
\begin{align}
\chi(z)=\int_{K_z}\widetilde{\delta}_z\chi dx\quad \forall \chi\in V(K_z),\quad |\widetilde{\delta}_z|_{W^{l,p}(\Omega)}\leq Ch^{-l-d(1-\frac{1}{p})}\quad \forall 1\leq p\leq \infty,~l=0,1,2,3,4.\label{rddf}
\end{align}

With $\delta_z$ and $\widetilde{\delta}_z$ at hand, we define the Green's function:
\begin{subequations}\label{Green}
\begin{align}
\boldsymbol{\mathcal{G}}(z,x,t)+\nabla G(z,x,t)&=0\quad \text{in}~\Omega\times (0,T],\\
\partial_t G(z,x,t)+\nabla\cdot \boldsymbol{\mathcal{G}}(z,x,t)&=0\quad \text{in}~\Omega\times(0,T],\\
G(z,x,t)&=0\quad \text{on}~\partial\Omega\times[0,T],\\
G(z,x,0)&=\delta_z\quad \text{in}~\Omega,
\end{align}
\end{subequations}
the regularized Green's function:
\begin{subequations}\label{rGreen}
\begin{align}
\boldsymbol{\mathcal{R}}(z,x,t)+\nabla\Gamma(z,x,t)&=0\quad \text{in}~\Omega\times (0,T],\label{rGreen:1}\\
\partial_t\Gamma(z,x,t)+\nabla\cdot\boldsymbol{\mathcal{R}}(z,x,t)&=0\quad \text{in}~\Omega\times(0,T],\label{rGreen:2}\\
\Gamma(z,x,t)&=0\quad \text{on}~\partial\Omega\times[0,T],\\
\Gamma(z,x,0)&=\widetilde{\delta}_z\quad \text{in}~\Omega,
\end{align}
\end{subequations}
and the semi-discrete scheme of (\ref{rGreen}): find $(\boldsymbol{\mathcal{R}}_h,\Gamma_h,\widehat{\Gamma}_h)\in \BW_h\times V_h\times\widehat{V}_h$ such that
\begin{subequations}\label{rGreenHDG}
\begin{align}
(\boldsymbol{\mathcal{R}}_h,\Bw_h)_{\mathcal{T}_h}-(\Gamma_h,\nabla\cdot\Bw_h)_{\mathcal{T}_h}+\langle\widehat{\Gamma}_h,\Bw_h\cdot\Bn\rangle
_{\partial\mathcal{T}_h}=0,&\label{rGreenHDG:1}\\
(\partial_t\Gamma_h,v_h)_{\mathcal{T}_h}-(\boldsymbol{\mathcal{R}}_h,\nabla v_h)_{\mathcal{T}_h}+\langle\widehat{\boldsymbol{\mathcal{R}}}_h\cdot
\Bn,v_h\rangle_{\partial\mathcal{T}_h}=0,&\label{rGreenHDG:2}\\
\langle\widehat{\boldsymbol{\mathcal{R}}}_h\cdot\Bn,\widehat{v}_h\rangle_{\partial\mathcal{T}_h}=0,&\label{rGreenHDG:3}\\
\widehat{\boldsymbol{\mathcal{R}}}_h\cdot\Bn=\boldsymbol{\mathcal{R}}_h\cdot\Bn+\tau(\mathcal{L}\Gamma_h-\widehat{\Gamma}_h),\quad \text{on}~\partial\mathcal{T}_h,&
\end{align}
\end{subequations}
for any $(\Bw_h,v_h,\widehat{v}_h)\in \BW_h\times V_h\times\widehat{V}_h$ and $\Gamma_h(z,x,0)=\Pi_V\tilde{\delta}_z$.

The Green's function $G(z,x,t)$ is symmetric with respect to $z$ and $x$ and has the following Gaussian pointwise estimate for the time derivatives
(cf. \cite[Appendix B with $\alpha=\beta=0$]{ge2006})
\begin{align}
|\partial_t^s G(z,x,t)|\leq Ct^{-s-\frac{d}{2}}e^{-\frac{|x-z|^2}{Ct}},\quad \forall z,x\in\Omega,~\forall t>0,~s=0,1,2,\cdots.\label{GE}
\end{align}
Since the regularized Green's function can be represented by (cf. \cite[(3.18)]{l2019})
\begin{align*}
\Gamma(z,x,t)=\int_{\Omega}G(x,y,t)\widetilde{\delta}_z(y) dy=\int_{\Omega}G(y,x,t)\widetilde{\delta}_z(y) dy,
\end{align*}
we know that the regularized Green's function $\Gamma$ admits time derivatives of any order. By finding time derivatives for (\ref{rGreen}), we have
\begin{align*}
\partial_{tt}\Gamma-\Delta\partial_t\Gamma=0\quad {\rm in}~\Omega\times (0,T],\quad \partial_t\Gamma=0\quad {\rm on}~\partial\Omega\times [0,T],\\
\partial_{ttt}\Gamma-\Delta\partial_{tt}\Gamma=0\quad {\rm in}~\Omega\times (0,T],\quad \partial_{tt}\Gamma=0\quad {\rm on}~\partial\Omega\times [0,T].
\end{align*}
For any $t\in (0,T]$, integrating the above two equations over the time interval $(0,t)$ yields
\begin{align*}
0=\int_0^t\partial_t(\partial_t\Gamma-\Delta\Gamma) dt=-\partial_t\Gamma(0)+\Delta\Gamma(0),\quad 0=\int_0^t\partial_t(\partial_{tt}\Gamma-\Delta\partial_t\Gamma) dt=-\partial_{tt}\Gamma(0)+\Delta\partial\Gamma(0).
\end{align*}
Then applying Theorem \ref{regularity} to the above two systems, we obtain $\partial_t\Gamma, \partial_{tt}\Gamma\in L^2((0,T);H^{1+\gamma}(\Omega))$.

Without lose generality, we assume throughout this section that $\text{diam}(\Omega)\leq 1$ and $T=1$, since the results for the case of $T>1$ can be obtained by parabolic re-scaling.
For any integer $j\geq 1$ and $z\in \overline{\Omega}$, we define $d_j=2^{-j}$ and
\begin{align*}
\Omega_j(z)=&\{x\in\Omega: d_j\leq |x-z|\leq d_{j-1}\},\\
Q_j(z)=&\{(t,x)\in Q_T: d_j\leq \max(|x-z|,\sqrt{t})\leq d_{j-1}\},\\
\Omega_{\ast}(z)=&\{x\in\Omega: |x-z|\leq Mh\},\\
Q_{\ast}(z)=&\{(t,x)\in Q_T: \max(|x-z|,\sqrt{t})\leq Mh\},
\end{align*}
where $Q_T=(0,T)\times \Omega$, $M>1$ is a constant, and $d_I\leq Mh\leq d_{I-1}$, thus $I\approx|\ln h|$.
In addition, we define
\begin{align*}
\Omega_j^1(z)=&\Omega_{j-1}(z)\cup\Omega_j(z)\cup\Omega_{j+1}(z),\\
\Omega_j^2(z)=&\Omega_{j-2}(z)\cup\Omega_j^1(z)\cup\Omega_{j+2}(z),\\
\Omega_j^3(z)=&\Omega_{j-3}(z)\cup\Omega_j^2(z)\cup\Omega_{j+3}(z),\\
Q_j^1(z)=&Q_{j-1}(z)\cup Q_j(z)\cup Q_{j+1}(z),\\
Q_{j}^2(z)=&Q_{j-2}(z)\cup Q_{j}^1(z)\cup Q_{j+2}(z),\\
Q_j^3(z)=&Q_{j-3}(z)\cup Q_{j}^2(z)\cup Q_{j+3}(z).
\end{align*}
Then according to the definition, it is easy to obtain that
\begin{align}
Q_T=\cup_{j=1}^I Q_j(z)\cup Q_{\ast}(z),\quad \Omega=\cup_{j=1}^I\Omega_j(z)\cup\Omega_{\ast}(z).\label{de}
\end{align}

Next we are going to prove some basic energy estimates for problems (\ref{rGreen}) and (\ref{rGreenHDG}).
To this end, we prove the following result.

\begin{lemma}\label{ee1}
Let $(\Bpsi_{h}, \phi_{h}, \widehat{\phi}_{h}) \in \BW_{h} \times V_{h} \times \widehat{V}_{h}$ satisfying
\begin{subequations}\label{ee1-eq}
\begin{align}
(\Bpsi_{h}, \Bw_h)_{\mathcal{T}_{h}} - (\phi_{h}, \nabla\cdot \Bw_h)_{\mathcal{T}_{h}}
+ \langle \widehat{\phi}_{h}, \Bw_h\cdot\Bn \rangle_{\partial\mathcal{T}_{h}} = 0,& \label{ee1-eq:1}\\
\langle \Bpsi_{h}\cdot\Bn + \tau(\mathcal{L}\phi_{h} - \widehat{\phi}_{h}), \widehat{v}_h
\rangle_{\partial\mathcal{T}_{h}} = 0,&\label{ee1-eq:2}
\end{align}
\end{subequations}
for any $(\Bw_h, \widehat{v}_h)\in \BW_{h} \times \widehat{V}_{h}$. Then we have
\begin{align}
\|\Bpsi_h\|_{L^2(\Omega)}+\Big(\sum_{K\in\mathcal{T}_h}\|\sqrt{\tau}(\mathcal{L}\phi_h-\widehat{\phi}_h)\|^2_{L^2(\partial K)}\Big)^{1/2}\leq Ch^{-1}\|\phi_h\|_{L^2(\Omega)}.
\end{align}
\end{lemma}
\begin{proof}
Setting $\Bw_h=\Bpsi_h$ in (\ref{ee1-eq:1}) to infer by (\ref{ee1-eq:2}), Lemma \ref{inv} and Lemma \ref{tra} that
\begin{align*}
\|\Bpsi_h\|^2_{L^2(\Omega)}+\sum_{K\in\mathcal{T}_h}&\|\sqrt{\tau}(\mathcal{L}\phi_h-\widehat{\phi}_h)\|^2_{L^2(\partial K)}=-(\nabla\phi_h,\Bpsi_h)_{\mathcal{T}_h}+
\langle\Bpsi_h\cdot\Bn+\tau(\mathcal{L}\phi_h-\widehat{\phi}_h),\phi_h\rangle_{\partial\mathcal{T}_h}\\
\leq&Ch^{-1}\|\phi_h\|_{L^2(\Omega)}\|\Bpsi_h\|_{L^2(\Omega)}+Ch^{-1}\|\phi_h\|_{L^2(\Omega)}\Big(\sum_{K\in\mathcal{T}_h}\|\sqrt{\tau}(\mathcal{L}\phi_h-\widehat{\phi}_h)\|^2_{L^2(\partial K)}\Big)^{1/2}.
\end{align*}
Hence we can derive the desired result by using Young's inequality to the above estimate.
\end{proof}

\begin{lemma}[energy estimates]\label{ee}
Let $(\boldsymbol{\mathcal{R}},\Gamma)$ and $(\boldsymbol{\mathcal{R}}_h,\Gamma_h,\widehat{\Gamma}_h)$ be the solutions of problems (\ref{rGreen}) and (\ref{rGreenHDG}). We have
\begin{align}
\|\boldsymbol{\mathcal{R}}\|_{L^2(Q_T)}+\|\nabla\Gamma\|_{L^2(Q_T)}+\|\boldsymbol{\mathcal{R}}_h\|_{L^2(Q_T)}+\|\nabla\Gamma_h\|_{L^2((0,T);L_h^2(\Omega))}&\nonumber\\
+\Big(\int_0^T\sum_{K\in\mathcal{T}_h}\|\sqrt{\tau}(\mathcal{L}\Gamma_h-\widehat{\Gamma}_h)\|^2_{L^2(\partial K)} dt\Big)^{1/2}\leq Ch^{-\frac{d}{2}},&\label{ee:1}\\
\|\partial_t\Gamma\|_{L^2(Q_T)}+\|\partial_t\Gamma_h\|_{L^2(Q_T)}\leq Ch^{-1-\frac{d}{2}},&\label{ee:2}\\
\|\partial_t\boldsymbol{\mathcal{R}}\|_{L^2(Q_T)}+\|\nabla\partial_t\Gamma\|_{L^2(Q_T)}+\|\partial_t\boldsymbol{\mathcal{R}}_h\|_{L^2(Q_T)}
+\|\nabla\partial_t\Gamma_h\|_{L^2((0,T);L_h^2(\Omega))}&\nonumber\\
+\Big(\int_0^T\sum_{K\in\mathcal{T}_h}\|\sqrt{\tau}\partial_t(\mathcal{L}\Gamma_h-\widehat{\Gamma}_h)\|^2_{L^2(\partial K)} dt\Big)^{1/2}\leq Ch^{-2-\frac{d}{2}},&\label{ee:3}\\
\|\partial_{tt}\Gamma\|_{L^2(Q_T)}+\|\partial_{tt}\Gamma_h\|_{L^2(Q_T)}\leq
Ch^{-3-\frac{d}{2}},&\label{ee:4}\\
\|\partial_{tt}\boldsymbol{\mathcal{R}}\|_{L^2(Q_T)}+\|\nabla\partial_{tt}\Gamma\|_{L^2(Q_T)}+\|\partial_{tt}\boldsymbol{\mathcal{R}}_h\|_{L^2(Q_T)}+
\|\nabla\partial_{tt}\Gamma_h\|_{L^2((0,T);L_h^2(\Omega))}&\label{ee:5}\\
+\Big(\int_0^T\sum_{K\in\mathcal{T}_h}\|\sqrt{\tau}\partial_{tt}(\mathcal{L}\Gamma_h-\widehat{\Gamma}_h)\|^2_{L^2(\partial K)}\Big)^{1/2}\leq Ch^{-4-\frac{d}{2}}.&\nonumber
\end{align}
\end{lemma}
\begin{proof}
Multiplying $\Gamma$ from both sides of (\ref{rGreen:2}) and setting $(\Bw_h,v_h,\widehat{v}_h)=(\boldsymbol{\mathcal{R}}_h,\Gamma_h,\widehat{\Gamma}_h)$ in (\ref{rGreenHDG}), we derive that
\begin{align}
\|\nabla \Gamma\|^2_{L^2(\Omega)}+(\partial_t\Gamma,\Gamma)+\|\boldsymbol{\mathcal{R}}_h\|^2_{L^2(\Omega)}+\sum_{K\in\mathcal{T}_h}\|\sqrt{\tau}(\mathcal{L}\Gamma_h-\widehat{\Gamma}_h)\|^2
_{L^2(\partial K)}+(\partial_t\Gamma_h,\Gamma_h)_{\mathcal{T}_h}=0.\label{ee-proof:ad3}
\end{align}
On the other hand, from Lemma \ref{H1} it is easy to observe that
\begin{align}
\|\nabla\Gamma_h\|_{L^2(K)}^2+h_K^{-1}\|\Gamma_h-\widehat{\Gamma}_h\|^2_{L^2(\partial K)}\leq C\big(
\|\boldsymbol{\mathcal{R}}_h\|^2_{L^2(K)}+\|\sqrt{\tau}(\mathcal{L}\Gamma_h-\widehat{\Gamma}_h)\|^2_{L^2(\partial K)}\big),\label{ee-proof:ad4}
\end{align}
for each $K\in\mathcal{T}_h$.
Thus integrating (\ref{ee-proof:ad3}) and (\ref{ee-proof:ad4}) over $[0,T]$ we obtain by (\ref{rddf}),
\begin{align}
&\|\boldsymbol{\mathcal{R}}\|_{L^2(Q_T)}+\|\nabla\Gamma\|_{L^2(Q_T)}+\|\boldsymbol{\mathcal{R}}_h\|_{L^2(Q_T)}+
\|\nabla\Gamma_h\|_{L^2((0,T);L_h^2(\Omega))}\nonumber\\
&+\Big(\int_0^T\sum_{K\in\mathcal{T}_h}\|\sqrt{\tau}(\mathcal{L}\Gamma_h-\widehat{\Gamma}_h)\|^2_{L^2(\partial K)} dt\Big)^{1/2}\leq C\big(\|\Gamma(0)\|_{L^2(\Omega)}+\|\Gamma_h(0)\|_{L^2(\Omega)}\big)\leq Ch^{-\frac{d}{2}}.\label{ee-proof:ad1}
\end{align}

Calculating the derivative with respect to the time from both sides of (\ref{rGreenHDG:1}) to infer that
\begin{align*}
(\partial_t\boldsymbol{\mathcal{R}}_h,\Bw_h)_{\mathcal{T}_h}+(\nabla\partial_t\Gamma_h,\Bw_h)_{\mathcal{T}_h}-\langle\partial_t(\Gamma_h-\widehat{\Gamma}_h),
\Bw_h\cdot\Bn\rangle_{\partial\mathcal{T}_h}=0,\quad \forall \Bw_h\in\BW_h.
\end{align*}
Then multiplying $\partial_t\Gamma$ from both sides of (\ref{rGreen:2}) and setting $(v_h,\widehat{v}_h)=(\partial_t\Gamma_h,\partial_t\widehat{\Gamma}_h)$
in (\ref{rGreenHDG:2}) and $\Bw_h=\boldsymbol{\mathcal{R}}_h$ in the above equation, we have
\begin{align*}
\|\partial_t\Gamma\|^2_{L^2(\Omega)}+(\nabla\Gamma,\partial_t\nabla\Gamma)+\|\partial_t\Gamma_h\|^2_{L^2(\Omega)}+
(\partial_t\boldsymbol{\mathcal{R}}_h,\boldsymbol{\mathcal{R}}_h)_{\mathcal{T}_h}+\langle\tau(\mathcal{L}\Gamma_h-\widehat{\Gamma}_h),
\partial_t(\mathcal{L}\Gamma_h-\widehat{\Gamma}_h)\rangle_{\partial\mathcal{T}_h}=0.
\end{align*}
Thus we can integrate the above result over $[0,T]$ to obtain by Lemma \ref{ee1} and (\ref{rddf}) that
\begin{align}
&\|\partial_t\Gamma\|_{L^2(Q_T)}+\|\partial_t\Gamma_h\|_{L^2(Q_T)}\nonumber\\
\leq& C\Big(\|\nabla\Gamma(0)\|_{L^2(\Omega)}+\|\boldsymbol{\mathcal{R}}_h(0)\|_{L^2(\Omega)}+\Big(\sum_{K\in\mathcal{T}_h}
\|\sqrt{\tau}(\mathcal{L}\Gamma_h-\widehat{\Gamma}_h)(0)\|^2_{L^2(\partial K)}\Big)^{1/2}\Big)\label{ee-proof:1}\\
\leq & C\big(\|\nabla\Gamma(0)\|_{L^2(\Omega)}+h^{-1}\|\Gamma_h(0)\|_{L^2(\Omega)}\big)\leq Ch^{-1-\frac{d}{2}}.\nonumber
\end{align}

By taking the derivative for (\ref{rGreen}) and (\ref{rGreenHDG}) with respect to the time, we can prove (\ref{ee:3}) with a similar method that has been
used to prove (\ref{ee:1}). According to (\ref{ee-proof:ad1}) we have
\begin{align}
&\|\partial_t\boldsymbol{\mathcal{R}}\|_{L^2(Q_T)}+\|\nabla\partial_t\Gamma\|_{L^2(Q_T)}+\|\partial_t\boldsymbol{\mathcal{R}}_h\|_{L^2(Q_T)}+
\|\nabla\partial_t\Gamma_h\|_{L^2((0,T);L_h^2(\Omega))}\nonumber\\
&+\Big(\int_0^T\sum_{K\in\mathcal{T}_h}\|\sqrt{\tau}\partial_t(\mathcal{L}\Gamma_h-\widehat{\Gamma}_h)\|^2_{L^2(\partial K)} dt\Big)^{1/2}\label{ee-proof:3}\\
\leq& C\big(\|\partial_t\Gamma(0)\|_{L^2(\Omega)}+\|\partial_t\Gamma_h(0)\|_{L^2(\Omega)}\big)=C\big(\|\Delta\Gamma(0)\|_{L^2(\Omega)}+
\|\partial_t\Gamma_h(0)\|_{L^2(\Omega)}\big).\nonumber
\end{align}
For any $v\in L^2(\Omega)$, by the definition of $\Pi_V$, (\ref{rGreenHDG:2}) and Lemma \ref{ee1} we have
\begin{align}
(\partial_t\Gamma_h(0),v)=&(\partial_t\Gamma_h(0),\Pi_Vv)_{\mathcal{T}_h}=(\boldsymbol{\mathcal{R}}_h(0),\nabla\Pi_Vv)_{\mathcal{T}_h}-\langle\boldsymbol{\mathcal{R}}_h(0)\cdot
\Bn+\tau(\mathcal{L}\Gamma_h-\widehat{\Gamma}_h)(0),\Pi_Vv\rangle_{\partial\mathcal{T}_h}\nonumber\\
\leq &Ch^{-1}\Big(\|\boldsymbol{\mathcal{R}}_h(0)\|_{L^2(\Omega)}+\Big(\sum_{K\in\mathcal{T}_h}\|\sqrt{\tau}(\mathcal{L}\Gamma_h-\widehat{\Gamma}_h)(0)\|^2_{L^2(\partial K)}\Big)^{1/2}\Big)\|\Pi_Vv\|_{L^2(\Omega)}\label{ee-proof:4}\\
\leq & Ch^{-2}\|\Gamma_h(0)\|_{L^2(\Omega)}\|v\|_{L^2(\Omega)},\nonumber
\end{align}
which, together with (\ref{ee-proof:3}) and (\ref{rddf}), results in
\begin{align*}
&\|\partial_t\boldsymbol{\mathcal{R}}\|_{L^2(Q_T)}+\|\nabla\partial_t\Gamma\|_{L^2(Q_T)}+\|\partial_t\boldsymbol{\mathcal{R}}_h\|_{L^2(Q_T)}+
\|\nabla\partial_t\Gamma_h\|_{L^2((0,T);L_h^2(\Omega))}\nonumber\\
&+\Big(\int_0^T\sum_{K\in\mathcal{T}_h}\|\sqrt{\tau}\partial_t(\mathcal{L}\Gamma_h-\widehat{\Gamma}_h)\|^2_{L^2(\partial K)} dt\Big)^{1/2}\\
\leq& C\big(\|\Delta\Gamma(0)\|_{L^2(\Omega)}+h^{-2}\|\Gamma_h(0)\|_{L^2(\Omega)}\big)\leq Ch^{-2-\frac{d}{2}}.
\end{align*}

Next, we turn to (\ref{ee:4}) and it can be proved by a similar method that has been used for (\ref{ee:2}). According to (\ref{ee-proof:1}), it is easy to obtain by (\ref{rddf}), (\ref{ee-proof:4}) and Lemma \ref{ee1} that
\begin{align*}
&\|\partial_{tt}\Gamma\|_{L^2(Q_T)}+\|\partial_{tt}\Gamma_h\|_{L^2(Q_T)}\\
\leq& C\Big(\|\nabla\partial_t\Gamma(0)\|_{L^2(\Omega)}
+\|\partial_t\boldsymbol{\mathcal{R}}_h(0)\|_{L^2(\Omega)}+\Big(\sum_{K\in\mathcal{T}_h}\|\sqrt{\tau}\partial_t(\mathcal{L}\Gamma_h-\widehat{\Gamma}_h)(0)\|^2_{L^2(\partial K)}\Big)^{1/2}\Big)\\
\leq& C\big(\|\nabla\partial_t\Gamma(0)\|_{L^2(\Omega)}+h^{-1}\|\partial_t\Gamma_h(0)\|_{L^2(\Omega)}\big)\\
\leq& C\big(\|\nabla\Delta\Gamma(0)\|_{L^2(\Omega)}+h^{-3}
\|\Gamma_h(0)\|_{L^2(\Omega)}\big)
\leq Ch^{-3-\frac{d}{2}}.
\end{align*}

Finally, we prove (\ref{ee:5}) by a similar method that has been used to prove (\ref{ee:3}). According to (\ref{ee-proof:3}), (\ref{ee-proof:4}), and (\ref{rddf}) we have
\begin{align}
&\|\partial_{tt}\boldsymbol{\mathcal{R}}\|_{L^2(Q_T)}+\|\nabla\partial_{tt}\Gamma\|_{L^2(Q_T)}+\|\partial_{tt}\boldsymbol{\mathcal{R}}_h\|_{L^2(Q_T)}+
\|\nabla\partial_{tt}\Gamma_h\|_{L^2((0,T);L_h^2(\Omega))}\nonumber\\
&+\Big(\int_0^T\sum_{K\in\mathcal{T}_h}\|\sqrt{\tau}\partial_{tt}(\mathcal{L}\Gamma_h-\widehat{\Gamma}_h)\|^2_{L^2(\partial K)}\Big)^{1/2}\nonumber\\
&\leq C\big(\|\partial_{tt}\Gamma(0)\|_{L^2(\Omega)}+\|\partial_{tt}\Gamma_h(0)\|_{L^2(\Omega)}\big)\nonumber\\
&\leq C\big(\|\Delta^2\Gamma(0)\|_{L^2(\Omega)}+h^{-2}\|\partial_t\Gamma_h(0)\|_{L^2(\Omega)}\big)\nonumber\\
&\leq C\big(\|\Delta^2\Gamma(0)\|_{L^2(\Omega)}+h^{-4}\|\Gamma_h(0)\|_{L^2(\Omega)}
\big)\leq Ch^{-4-\frac{d}{2}}.\nonumber
\end{align}
\end{proof}

In addition to the energy estimates presented in Lemma \ref{ee}, we also need to prove some local properties for $\Gamma_h$ and $\boldsymbol{\mathcal{R}}_h$
at $t=0$.

\begin{lemma}\label{ee3}
Let $(\boldsymbol{\mathcal{R}}_h,\Gamma_h,\widehat{\Gamma}_h)\in \BW_h\times V_h\times\widehat{V}_h$ be the solution of problem (\ref{rGreenHDG}), then for
$j=1,2,\cdots,I$ we have
\begin{align}
&h\|\partial_t\Gamma_h(0)\|_{L_h^2(\Omega_j(z))}+\|\boldsymbol{\mathcal{R}}_h(0)\|_{L_h^2(\Omega_j(z))}
+\Big(\sum_{K\in\mathcal{T}_h(\Omega_j(z))}\|\sqrt{\tau}(\mathcal{L}\Gamma_h-\widehat{\Gamma}_h)(0)\|^2_{L^2(\partial K)}\Big)^{1/2}\nonumber\\
&\leq C(hd_j^{-1})^{\beta_3}\Big(\|\boldsymbol{\mathcal{R}}_h(0)\|_{L^2_h(\mathcal{T}_h(\Omega_j^1(z)))}
+\Big(\sum_{K\in\mathcal{T}_h(\Omega_j^{1}(z))}\|\sqrt{\tau}(\mathcal{L}\Gamma_h-\widehat{\Gamma}_h)(0)\|^2_{L^2(\partial K)}\Big)^{1/2}\Big),
\end{align}
for any real number $\beta_3>0$, where $C$ is a positive constant depending on $\beta_3$.
\end{lemma}
\begin{proof}
Let $\varphi\in\mathcal{C}_0^{\infty}(\Omega_j^{1}(z))$ with $\varphi\equiv1$ on $\Omega_j(z)$. We define $(\varphi_h|_K,\widehat{\varphi}_h|_F)\in\mathcal{P}^0(K)\times \mathcal{P}^0(F)$ as
$\varphi_h|_K=\sup_{x\in K}\varphi(x)$ and $\widehat{\varphi}_h|_F=\sup_{x\in F}\varphi(x)$ for $K\in\mathcal{T}_h$ and $F\in\mathcal{E}_h$. Then by setting $\Bw_h=\varphi_h^2\boldsymbol{\mathcal{R}}_h(0)$
in (\ref{rGreenHDG:1}) and using (\ref{ee-proof:ad4}) and the fact that $\Gamma_h(0)=\Pi_V\tilde{\delta}_z=0$ in $\Omega_j^{1}(z)$, we have
\begin{align}
&\|\varphi_h\boldsymbol{\mathcal{R}}_h(0)\|^2_{L^2(\Omega)}+\sum_{K\in\mathcal{T}_h}\|\sqrt{\tau}\varphi_h(\mathcal{L}\Gamma_h-\widehat{\Gamma}_h)(0)\|^2_{L^2(\partial K)}\nonumber\\
=&\langle\boldsymbol{\mathcal{R}}_h(0)\cdot\Bn+\tau(\mathcal{L}\Gamma_h-\widehat{\Gamma}_h)(0),\varphi_h^2(\Gamma_h-\widehat{\Gamma}_h)(0)\rangle_{\partial\mathcal{T}_h}\nonumber\\
=&\langle\boldsymbol{\mathcal{R}}_h(0)\cdot\Bn+\tau(\mathcal{L}\Gamma_h-\widehat{\Gamma}_h)(0),(\varphi_h^2-\widehat{\varphi}_h^2)(\Gamma_h-\widehat{\Gamma}_h)(0)\rangle_{\partial\mathcal{T}_h}\nonumber\\
\leq& Chd_j^{-1}\sum_{K\in\mathcal{T}_h}\Big(\|\sqrt{\tau}\varphi_h(\mathcal{L}\Gamma_h-\widehat{\Gamma}_h)(0)\|_{L^2(\partial K)}\|\sqrt{\tau}(\Gamma_h-\widehat{\Gamma}_h)(0)\|_{L^2(\partial K)}\nonumber\\
&+h_K^{-1/2}\|(\Gamma_h-\widehat{\Gamma}_h)(0)\|_{L^2(\partial K)}\|\varphi_h\boldsymbol{\mathcal{R}}_h(0)\|_{L^2(K)}\Big)\nonumber\\
\leq &\frac{1}{2}\Big(\|\varphi_h\boldsymbol{\mathcal{R}}_h(0)\|_{L^2(\Omega)}^2+\sum_{K\in\mathcal{T}_h}\|\sqrt{\tau}\varphi_h(\mathcal{L}\Gamma_h-
\widehat{\Gamma}_h)(0)\|_{L^2(\partial K)}^2\Big)\nonumber\\
&+C(hd_j^{-1})^2\Big(\|\boldsymbol{\mathcal{R}}_h(0)\|^2_{L^2_h(\mathcal{T}_h(\Omega_j^1(z)))}
+\sum_{K\in\mathcal{T}_h(\Omega_j^1(z))}\|\sqrt{\tau}(\mathcal{L}\Gamma_h-\widehat{\Gamma}_h)(0)\|^2
_{L^2(\partial K)}\Big),\nonumber
\end{align}
which, together with (\ref{ee-proof:4}), results in
\begin{align*}
&h\|\varphi_h\partial_t\Gamma_h(0)\|_{L^2(\Omega)}+\|\varphi_h\boldsymbol{\mathcal{R}}_h(0)\|_{L^2(\Omega)}
+\Big(\sum_{K\in\mathcal{T}_h}\|\sqrt{\tau}\varphi_h(\mathcal{L}\Gamma_h-\widehat{\Gamma}_h)(0)\|^2_{L^2(\partial K)}\Big)^{1/2}\\
&\leq Chd_j^{-1}\Big(\|\boldsymbol{\mathcal{R}}_h(0)\|_{L^2_h(\mathcal{T}_h(\Omega_j^1(z)))}
+\Big(\sum_{K\in\mathcal{T}_h(\Omega_j^{1}(z))}\|\sqrt{\tau}(\mathcal{L}\Gamma_h-\widehat{\Gamma}_h)(0)\|^2_{L^2(\partial K)}\Big)^{1/2}\Big).
\end{align*}
Thus, applying the proof method from Lemma \ref{local_energy_estimate3}, we obtain that for any real number $\beta_3>0$
\begin{align*}
&h\|\partial_t\Gamma_h(0)\|_{L_h^2(\Omega_j(z))}+\|\boldsymbol{\mathcal{R}}_h(0)\|_{L_h^2(\Omega_j(z))}
+\Big(\sum_{K\in\mathcal{T}_h(\Omega_j(z))}\|\sqrt{\tau}(\mathcal{L}\Gamma_h-\widehat{\Gamma}_h)(0)\|^2_{L^2(\partial K)}\Big)^{1/2}\\
&\leq C(hd_j^{-1})^{\beta_3}\Big(\|\boldsymbol{\mathcal{R}}_h(0)\|_{L^2_h(\mathcal{T}_h(\Omega_j^1(z)))}
+\Big(\sum_{K\in\mathcal{T}_h(\Omega_j^{1}(z))}\|\sqrt{\tau}(\mathcal{L}\Gamma_h-\widehat{\Gamma}_h)(0)\|^2_{L^2(\partial K)}\Big)^{1/2}\Big),
\end{align*}
where $C$ is a positive constant depending on $\beta_3$.
\end{proof}

Corresponding to the regularized Green's function, from Lemma 4.1 and its proof in \cite{l2019} we further have the following local energy estimates.

\begin{lemma}\label{rgge}
For the regularized Green's function $\Gamma$ defined in (\ref{rGreen}), we have
\begin{align}
d_j^{-4}\|\Gamma\|_{L^2H^{1+\gamma}(Q_j(z))}+d_j^{-2}\|\partial_t\Gamma\|_{L^2H^{1+\gamma}(Q_j(z))}+\|\partial_{tt}\Gamma\|_{L^2H^{1+\gamma}(Q_j(z))}
\leq& Cd_j^{-4-\gamma-\frac{d}{2}},\\
d_j^{-2}\|\Gamma\|_{L^{\infty}H^{1+\gamma}(Q_j(z))}+\|\partial_t\Gamma\|_{L^{\infty}H^{1+\gamma}(Q_j(z))}\leq& Cd_j^{-\gamma-3-\frac{d}{2}},
\end{align}
where
\begin{align*}
&\|v\|_{L^2(Q_j(z))}^2=\int_0^{d_j^2}\|v\|_{L_h^2(\Omega_j(z))}^2 dt+
\int_{d_j^2}^{d_{j-1}^2}\|v\|^2_{L_h^2(B_{d_{j-1}}(z))} dt,\\
&\|v\|^2_{L^2H^{1+\gamma}(Q_j(z))}=\int_0^{d_j^2}\|v\|_{H^{1+\gamma}_h(\Omega_j(z))}^2 dt+
\int_{d_j^2}^{d_{j-1}^2}\|v\|^2_{H^{1+\gamma}_h(B_{d_{j-1}}(z))} dt,\\
&\|v\|_{L^{\infty}H^{1+\gamma}(Q_j(z))}=\max_t\|v\|_{H^{1+\gamma}_h(Q_j^t(z))},\quad Q_j^t(z)=\{x\in\Omega, (x,t)\in Q_j(z)\},
\end{align*}
$\Omega_j(z)\times [0,d_j^2]$ and $B_{d_{j-1}}(z)\times [d_j^2, d_{j-1}^2]$ are two cylinders consisting of the set $Q_j(z)$, and $B_{d_{j-1}}(z)=\{x\in\Omega: |x-z|\leq d_{j-1}\}$.
\end{lemma}

With all ingredients at hand, now we are going to prove the maximum-norm stability. Before this, we present an expression of $u_h(z,t)$ for any $z\in K_z$ with $K_z\in\mathcal{T}_h$ and $t>0$ in the following lemma.

\begin{lemma}\label{prep}
Let $(\Bq_h,u_h,\widehat{u}_h)$ and $(\boldsymbol{\mathcal{R}}_h,\Gamma_h,\widehat{\Gamma}_h)$ be the solutions of problems (\ref{HDG}) and (\ref{rGreenHDG}),
then we have for any $z\in K_z$ with $K_z\in\mathcal{T}_h$ and $t>0$
\begin{align}
u_h(z,t)=\int_0^t\int_{\Omega}f(x,s)\Gamma_h(z,x,t-s) dxds+\int_{\Omega}u_{0,h}\Gamma_h(z,x,t) dx,
\end{align}
where $u_{0,h}(x)=u_h(x,0)$ is an approximation of $u_0$ in $V_h$.
\end{lemma}
\begin{proof}
According to (\ref{rddf}), we have
\begin{align}
u_h(z,t)=&\int_{\Omega}u_h(x,t)\tilde{\delta}_z dx=\int_{\Omega}u_h(x,t)\Pi_V\tilde{\delta}_z dx=\int_{\Omega}u_h(x,t)\Gamma_h(z,x,0) dx\nonumber\\
=&\int_{\Omega}\int_0^t\frac{\partial}{\partial s}\big(u_h(x,s)\Gamma_h(z,x,t-s)\big) ds dx+\int_{\Omega}u_{0,h}\Gamma_h(z,x,t) dx\label{prep-proof:1}\\
=&\int_0^t\int_{\Omega}\Gamma_h(z,x,t-s)\partial_su_h(x,s) dx ds+\int_0^t\int_{\Omega}u_h(x,s)\partial_s\Gamma_h(z,x,t-s) dxds\nonumber\\
&+\int_{\Omega}u_{0,h}\Gamma_h(z,x,t) dx.\nonumber
\end{align}
Next we are going to calculate the first and second terms on the right hand side of (\ref{prep-proof:1}) separately. By (\ref{HDG:2}), (\ref{HDG:3}),
(\ref{rGreenHDG:1}) and integration by parts, we have
\begin{align}
&\int_0^t\int_{\Omega}\Gamma_h(z,x,t-s)\partial_su_h(x,s) dx ds\nonumber\\
=&\int_0^t(f,\Gamma_h(t-s)) ds+\int_0^t-(\Gamma_h(t-s),\nabla\cdot\Bq_h)_{\mathcal{T}_h}-\langle\tau(\mathcal{L}u_h-\widehat{u}_h),\Gamma_h(t-s)\rangle_{\partial\mathcal{T}_h} ds\nonumber\\
=&\int_0^t(f,\Gamma_h(t-s)) ds+\int_0^t -(\boldsymbol{\mathcal{R}}_h(t-s),\Bq_h)_{\mathcal{T}_h}-\langle\widehat{\Gamma}_h(t-s),\Bq_h\cdot\Bn\rangle_{\partial\mathcal{T}_h}
-\langle\tau(\mathcal{L}u_h-\widehat{u}_h),\Gamma_h(t-s)\rangle_{\partial\mathcal{T}_h} ds\nonumber\\
=&\int_0^t(f,\Gamma_h(t-s)) ds-\int_0^t(\boldsymbol{\mathcal{R}}_h(t-s),\Bq_h)_{\mathcal{T}_h}+\langle\tau(\mathcal{L}u_h-\widehat{u}_h),(\Gamma_h-\widehat{\Gamma}_h)
(t-s)\rangle_{\partial\mathcal{T}_h} ds.\nonumber
\end{align}
Similarly, by (\ref{rGreenHDG:2}), (\ref{rGreenHDG:3}), (\ref{HDG:1}) and integration by parts, we have
\begin{align}
&\int_0^t\int_{\Omega}u_h(x,s)\partial_s\Gamma_h(z,x,t-s) dxds\nonumber\\
=&-\int_0^t\int_{\Omega}u_h(x,t-s)\partial_s\Gamma_h(z,x,s) dx ds\nonumber\\
=&\int_0^t(u_h(t-s),\nabla\cdot\boldsymbol{\mathcal{R}}_h)_{\mathcal{T}_h}+\langle\tau(\mathcal{L}\Gamma_h-\widehat{\Gamma}_h),u_h(t-s)\rangle_{\partial\mathcal{T}_h} ds\nonumber\\
=&\int_0^t (\Bq_h(t-s),\boldsymbol{\mathcal{R}}_h)_{\mathcal{T}_h}+\langle\widehat{u}_h(t-s),\boldsymbol{\mathcal{R}}_h\cdot\Bn\rangle_{\partial\mathcal{T}_h}
+\langle\tau(\mathcal{L}\Gamma_h-\widehat{\Gamma}_h),u_h(t-s)\rangle_{\partial\mathcal{T}_h} ds\nonumber\\
=&\int_0^t(\Bq_h(t-s),\boldsymbol{\mathcal{R}}_h)_{\mathcal{T}_h}+\langle\tau(\mathcal{L}\Gamma_h-\widehat{\Gamma}_h),(u_h-\widehat{u}_h)(t-s)\rangle_{\partial\mathcal{T}_h} ds\nonumber\\
=&\int_0^t(\Bq_h,\boldsymbol{\mathcal{R}}_h(t-s))_{\mathcal{T}_h}+\langle\tau(\mathcal{L}\Gamma_h-\widehat{\Gamma}_h)(t-s),(u_h-\widehat{u}_h)\rangle_{\partial\mathcal{T}_h} ds.\nonumber
\end{align}
Thus we derive the desired result by using the above three equalities.
\end{proof}

\begin{theorem}[maximum-norm stability]\label{max-stability}
Let $u$ and $(\Bq_h,u_h,\widehat{u}_h)$ be the solutions of problems (\ref{model}) and (\ref{HDG}), we have
\begin{align}
\|u_h\|_{L^{\infty}(\Omega)}\leq C\|f\|_{L^{\infty}((0,t);L^{\infty}(\Omega))}+C\|u_{0,h}\|_{L^{\infty}(\Omega)},\quad \forall t\in(0,T].\label{max-stability:1}
\end{align}
If we further have $f=0$, then the following maximum-norm stability of the discrete semigroup holds
\begin{align}
\|u_h\|_{L^{\infty}(\Omega)}+t\|\partial_t u_h\|_{L^{\infty}(\Omega)}\leq C\|u_{0,h}\|_{L^{\infty}(\Omega)},\quad \forall t\in (0,T].\label{max-stability:2}
\end{align}
Here $u_{0,h}\in V_h$ is an approximation of $u_0$, and the constants $C$ in
(\ref{max-stability:1}) and (\ref{max-stability:2}) are both independent of $T$.
\end{theorem}

\emph{Proof of (\ref{max-stability:2})}:
By Lemma \ref{prep} and the definitions of $\Gamma$ and $\Gamma_h$, we have for any $z\in K_z$ with $K_z\in\mathcal{T}_h$ and $t>0$
\begin{align}
u_h(z,t)=&\int_{\Omega}u_{0,h}(\Gamma_h-\Gamma) dx+\int_{\Omega}u_{0,h}\Gamma dx\nonumber\\
=&\int_0^t\int_{\Omega}u_{0,h}\partial_s(\Gamma_h-\Gamma) dxds+\int_{\Omega}u_{0,h}\Gamma dx\label{max-stability-proof:-1}\\
\leq& \|u_{0,h}\|_{L^{\infty}(\Omega)}\big(\|\partial_s(\Gamma_h-\Gamma)\|_{L^1((0,t);L^1(\Omega))}+\|\Gamma\|_{L^1(\Omega)}\big),\nonumber
\end{align}
and
\begin{align}
t\partial_t u_h(z,t)=&\int_{\Omega}u_{0,h}t\partial_t(\Gamma_h-\Gamma) dx+\int_{\Omega}u_{0,h}t\partial_t\Gamma dx\nonumber\\
=&\int_{\Omega}\int_0^t\big(s\partial_{ss}(\Gamma_h-\Gamma)+\partial_s(\Gamma_h-\Gamma)\big)u_{0,h} dsdx+\int_{\Omega}u_{0,h}t\partial_t\Gamma dx\nonumber\\
\leq&\|u_{0,h}\|_{L^{\infty}(\Omega)}\big(\|s\partial_{ss}(\Gamma_h-\Gamma)\|_{L^1((0,t);L^1(\Omega))}+\|\partial_s(\Gamma_h-\Gamma)\|_{L^1((0,t);L^1(\Omega))}
+t\|\partial_t\Gamma\|_{L^1(\Omega)}\big).\nonumber
\end{align}
By the analyticity of the continuous parabolic semigroup on $L^1(\Omega)$ (cf. \cite{tmw2006}), we know that
\begin{align}
\|\Gamma\|_{L^1(\Omega)}+t\|\partial_t\Gamma\|_{L^1(\Omega)}\leq C\|\Gamma(0)\|_{L^1(\Omega)}\leq C.\label{max-stability-proof:0}
\end{align}
Thus our remaining task is to prove
\begin{align}
\|\partial_s(\Gamma_h-\Gamma)\|_{L^1((0,t);L^1(\Omega))}+\|s\partial_{ss}(\Gamma_h-\Gamma)\|_{L^1((0,t);L^1(\Omega))}\leq C.\label{max-stability-proof:01}
\end{align}

To this end, by (\ref{de}) we have
\begin{align*}
&\|\partial_t(\Gamma_h-\Gamma)\|_{L^1((0,T);L^1(\Omega))}+\|t\partial_{tt}(\Gamma_h-\Gamma)\|_{L^1((0,T);L^1(\Omega))}\\
\leq& C(Mh)^{1+\frac{d}{2}}\big(\|\partial_t(\Gamma_h-\Gamma)\|_{L^2(Q_{\ast}(z))}+\|t\partial_{tt}(\Gamma_h-\Gamma)\|_{L^2(Q_{\ast}(z))}\big)\\
&+C\sum_{j=1}^Id_j^{1+\frac{d}{2}}\big(\|\partial_t(\Gamma_h-\Gamma)\|_{L^2(Q_{j}(z))}+\|t\partial_{tt}(\Gamma_h-\Gamma)\|_{L^2(Q_{j}(z))}\big),
\end{align*}
which, together with Lemma \ref{ee}, results in
\begin{align}
\|\partial_t(\Gamma_h-\Gamma)\|_{L^1((0,T);L^1(\Omega))}+\|t\partial_{tt}(\Gamma_h-\Gamma)\|_{L^1((0,T);L^1(\Omega))}
\leq C+C\sum_{j=1}^Id_j^{1+\frac{d}{2}}\mathcal{M}_j,\label{max-stability-proof:1}
\end{align}
where
\begin{align*}
\mathcal{M}_j=&d_j^{-1}\Big(\|\boldsymbol{\mathcal{R}}-\boldsymbol{\mathcal{R}}_h\|_{L^2(Q_j(z))}+\|\nabla(\Gamma-\Gamma_h)\|_{L^2(Q_j(z))}+\interleave(\Gamma_h,\widehat{\Gamma}_h)\interleave_{L^2(Q_j(z))}
\Big)\\
&+\|\partial_t(\Gamma-\Gamma_h)\|_{L^2(Q_j(z))}+d_j^2\|\partial_{tt}(\Gamma-\Gamma_h)\|_{L^2(Q_j(z))},\\
\interleave(v,&\mu)\interleave^2_{L^2(Q_j(z))}=\int_0^{d_j^2}\sum_{K\in\mathcal{T}_h}\|\sqrt{\tau}(\mathcal{L}v-\mu)\|^2_{L^2(\partial K\cap\Omega_j(z))}dt\\
&\quad\quad\quad\quad\quad\quad+\int_{d_j^2}^{d_{j-1}^2}\sum_{K\in\mathcal{T}_h}\|\sqrt{\tau}(\mathcal{L}v-\mu)\|^2_{L^2(\partial K\cap B_{d_{j-1}})}dt.
\end{align*}
Now we proceed to estimate the last term on the right hand side of (\ref{max-stability-proof:1}). By replacing $({\bm q},u,{\bm q}_h,u_h,\widehat{u}_h)$ with $(\boldsymbol{\mathcal{R}},\Gamma,\boldsymbol{\mathcal{R}}_h,\Gamma_h,\widehat{\Gamma}_h)$ and setting $\beta_1=\beta_2=\beta$ in Lemma \ref{local_energy_estimate3}, we obtain
\begin{align}
\mathcal{M}_j\leq C\big(I_j+X_j+H_j+d_j^{-2}\|\Gamma-\Gamma_h\|_{L^2(Q_j^1(z))}\big),\label{max-stability-proof:2}
\end{align}
where
\begin{align*}
I_j=&d_j^{-1}\|(\Gamma-\Gamma_h)(0)\|_{L^2(\Omega_j^1(z))}+d_j\|\partial_t(\Gamma-\Gamma_h)(0)\|_{L^2(\Omega_j^1(z))}+hd_j^2\|\partial_{tt}(\Gamma-
\Gamma_h)(0)\|_{L_h^2(\Omega_j^1(z))}\\
&+\|(\boldsymbol{\mathcal{R}}-\boldsymbol{\mathcal{R}}_h)(0)\|_{L^2(\Omega_j^1(z))}
+d_j^2\|\partial_t(\boldsymbol{\mathcal{R}}-\boldsymbol{\mathcal{R}}_h)(0)\|_{L^2(\Omega_j^1(z))}\\
&+\Big(\sum_{K\in\mathcal{T}_h(\Omega^1_j(z))}\|\sqrt{\tau}(\mathcal{L}\Gamma_h-\widehat{\Gamma}_h)(0)\|^2_{L^2(\partial K)}\Big)^{1/2}+d_j^2\Big(\sum_{K\in\mathcal{T}_h(\Omega^1_j(z))}\|\sqrt{\tau}\partial_t(\mathcal{L}\Gamma_h-\widehat{\Gamma}_h)(0)\|^2_{L^2(\partial K)}\Big)^{1/2}\\
&+h^{-1}\|\delta_{\Gamma}(0)\|_{L^2(\Omega_j^1(z))}+\|\nabla\delta_{\Gamma}(0)\|_{L^2(\Omega_j^1(z))}
+h^{-1}d_j^2\|\partial_t\delta_{\Gamma}(0)\|_{L^2(\Omega_j^1(z))}+hd_j^2\|\partial_{tt}\delta_{\Gamma}(0)\|_{L_h^2(\Omega_j^1(z))}\\
&+d_j^2\|\nabla\partial_t\delta_{\Gamma}(0)\|_{L^2(\Omega_j^1(z))}+\|\delta_{\boldsymbol{\mathcal{R}}}(0)\|_{L^2(\Omega_j^1(z))}
+d_j^2\|\partial_t\delta_{\boldsymbol{\mathcal{R}}}(0)\|_{L^2(\Omega_j^1(z))},\\
X_j=&h^{-1}d_j^{-1}\|\delta_{\Gamma}\|_{L^2(Q_j^1(z))}+d_j^{-1}\|\nabla\delta_{\Gamma}\|_{L^2(Q_j^1(z))}+h^{-1}d_j\|\partial_t\delta_{\Gamma}\|_{L^2(Q_j^1(z))}
+d_j\|\nabla\partial_t\delta_{\Gamma}\|_{L^2(Q_j^1(z))}\\
&+h^{-1}d_j^3\|\partial_{tt}\delta_{\Gamma}\|_{L^2(Q_j^1(z))}+d_j^3\|\nabla\partial_{tt}\delta_{\Gamma}\|_{L^2(Q_j^1(z))}+d_j^{-1}
\|\delta_{\boldsymbol{\mathcal{R}}}\|_{L^2(Q_j^1(z))}+h^{\gamma}d_j^{-1}\|\delta_{\boldsymbol{\mathcal{R}}}\|_{L^2H^{\gamma}(Q_j^1(z))} \\
&+d_j\|\partial_t\delta_{\boldsymbol{\mathcal{R}}}\|_{L^2(Q_j^1(z))}+h^{\gamma}d_j\|\partial_t\delta_{\boldsymbol{\mathcal{R}}}\|_{L^2H^{\gamma}(Q_j^1(z))}+
d_j^3\|\partial_{tt}\delta_{\boldsymbol{\mathcal{R}}}\|_{L^2(Q_j^1(z))}+h^{\gamma}d_j^3\|\partial_{tt}\delta_{\boldsymbol{\mathcal{R}}}\|_{L^2H^{\gamma}(Q_j^1(z))}\\
&+h^{-1}\|\delta_{\Gamma}\|_{L^{\infty}L^2(Q_j^1(z))}+\|\nabla\delta_{\Gamma}\|_{L^{\infty}L^2(Q_j^1(z))}
+h^{-1}d_j^2\|\partial_t\delta_{\Gamma}\|_{L^{\infty}L^2(Q_j^1(z))}\\
&+d_j^2\|\nabla\partial_t\delta_{\Gamma}\|_{L^{\infty}L^2(Q_j^1(z))}+\|\delta_{\boldsymbol{\mathcal{R}}}\|_{L^{\infty}L^2(Q_j^1(z))}
+h^{\gamma}\|\delta_{\boldsymbol{\mathcal{R}}}\|_{L^{\infty}H^{\gamma}(Q_j^1(z))} \\
&+d_j^2\|\partial_t\delta_{\boldsymbol{\mathcal{R}}}\|_{L^{\infty}L^2(Q_j^1(z))}
+h^{\gamma}d_j^2\|\partial_t\delta_{\boldsymbol{\mathcal{R}}}\|_{L^{\infty}H^{\gamma}(Q_j^1(z))}, \\
H_j=& (hd_j^{-1})^{\beta}\Big(d_j^{-1}\big(\|\boldsymbol{\mathcal{R}}-\boldsymbol{\mathcal{R}}_h\|_{L^2(Q_j^1(z))}
+\interleave(\Gamma_h,\widehat{\Gamma}_h)\interleave_{L^2(Q_j^1(z))}\big)\\
&+ d_j\big(\|\partial_t(\boldsymbol{\mathcal{R}}-\boldsymbol{\mathcal{R}}_h)\|_{L^2(Q_j^1(z))}+\interleave(\partial_{t}\Gamma_h,\partial_{t}
\widehat{\Gamma}_h)\interleave_{L^2(Q_j^1(z))}\big)\\
&+ hd_j^2\big(\|\partial_{tt}(\boldsymbol{\mathcal{R}}-\boldsymbol{\mathcal{R}}_h)\|_{L^2(Q_j^1(z))}+\interleave(\partial_{tt}\Gamma_h,\partial_{tt}
\widehat{\Gamma}_h)\interleave_{L^2(Q_j^1(z))}\big)\Big).
\end{align*}

Since $\Gamma(0)=\tilde{\delta}_z=0$ and $\Gamma_h(0)=\Pi_V\tilde{\delta}_z=0$ in $\Omega_j^1(z)$, according to the definition of $\Gamma$ we have
\begin{align*}
\partial_t\Gamma(0)=\Delta\Gamma(0)=0,\quad \partial_{tt}\Gamma(0)=\Delta\partial_t\Gamma(0)=0,\quad \text{in}~\Omega_j^1(z),
\end{align*}
which, together with Lemma \ref{ee3} and the fact that $\boldsymbol{\mathcal{R}}(0)=-\nabla\Gamma(0)$, arrives at
\begin{align}
I_j=&d_j\|\partial_t\Gamma_h(0)\|_{L^2(\Omega_j^1(z))}+hd_j^2\|\partial_{tt}
\Gamma_h(0)\|_{L_h^2(\Omega_j^1(z))}+\|\boldsymbol{\mathcal{R}}_h(0)\|_{L^2(\Omega_j^1(z))}
+d_j^2\|\partial_t\boldsymbol{\mathcal{R}}_h(0)\|_{L^2(\Omega_j^1(z))}\nonumber\\
&+\Big(\sum_{K\in\mathcal{T}_h(\Omega^1_j(z))}\|\sqrt{\tau}(\mathcal{L}\Gamma_h-\widehat{\Gamma}_h)(0)\|^2_{L^2(\partial K)}\Big)^{1/2}+d_j^2\Big(\sum_{K\in\mathcal{T}_h(\Omega^1_j(z))}\|\sqrt{\tau}\partial_t(\mathcal{L}\Gamma_h-\widehat{\Gamma}_h)(0)\|^2_{L^2(\partial K)}\Big)^{1/2}\nonumber\\
\leq&C(hd_j^{-1})^{\beta_3-1}\Big(\|\boldsymbol{\mathcal{R}}_h(0)\|_{L_h^2(\Omega_j^2(z))}
+\Big(\sum_{K\in\mathcal{T}_h(\Omega_j^2(z))}\|\sqrt{\tau}(\mathcal{L}\Gamma_h-\widehat{\Gamma}_h)(0)\|^2_{L^2(\partial K)}\Big)^{1/2}\Big)\nonumber\\
&+C(hd_j^{-1})^{\beta_3}d_j^2\Big(\|\partial_t\boldsymbol{\mathcal{R}}_h(0)\|_{L_h^2(\Omega_j^2(z))}
+\sum_{K\in\mathcal{T}_h(\Omega_j^2(z))}\|\sqrt{\tau}\partial_t(\mathcal{L}\Gamma_h-\widehat{\Gamma}_h)(0)\|^2_{L^2(\partial K)}\Big)^{1/2}.\nonumber
\end{align}
By choosing $\beta_3=4+\frac{d}{2}$, we obtain by (\ref{rddf}), (\ref{ee-proof:4}) and Lemma \ref{ee1} that
\begin{align}
I_j\leq C(hd_j^{-1})^{3+\frac{d}{2}}h^{-1-\frac{d}{2}}+C(hd_j^{-1})^{4+\frac{d}{2}}d_j^2h^{-3-\frac{d}{2}}\leq Chd_j^{-2-\frac{d}{2}}.\label{max-stability-proof:3}
\end{align}

On the other hand, with the interpolation error estimates and Lemma \ref{rgge} we have
\begin{align}
X_j\leq& C\big(h^{\gamma}d_j^{-1}\|\Gamma\|_{L^2H^{1+\gamma}(Q_j^2(z))}+h^{\gamma}d_j\|\partial_t\Gamma\|_{L^2H^{1+\gamma}(Q_j^2(z))}
+h^{\gamma}d_j^3\|\partial_{tt}\Gamma\|_{L^2H^{1+\gamma}(Q_j^2(z))}\big)\nonumber\\
&+C\big(h^{\gamma}\|\Gamma\|_{L^{\infty}H^{1+\gamma}(Q_j^2(z))}+h^{\gamma}d_j^2\|\partial_t\Gamma\|_{L^{\infty}H^{1+\gamma}(Q_j^2(z))}
\leq Ch^{\gamma}d_j^{-\gamma-1-\frac{d}{2}},\label{max-stability-proof:4}
\end{align}
Next we are going to deal with $H_j$. With Lemma \ref{ee} we have by setting $\beta=4+\frac{d}{2}$
\begin{align}
H_j\leq C(hd_j^{-1})^{\beta}\big(d_j^{-1}h^{-\frac{d}{2}}+d_jh^{-2-\frac{d}{2}}+hd_j^2h^{-4-\frac{d}{2}}\big)\leq Chd_j^{-2-\frac{d}{2}}.\label{max-stability-proof:5}
\end{align}

Thus collecting (\ref{max-stability-proof:2})-(\ref{max-stability-proof:5}) together we can derive that
\begin{align}
\sum_{j=1}^I d_j^{1+\frac{d}{2}}\mathcal{M}_j
\leq & C\sum_{j=1}^I\big(h^{\gamma}d_j^{-\gamma}+d_j^{-1+\frac{d}{2}}\|\Gamma-\Gamma_h\|_{L^2(Q_j^1(z))}\big)\leq C+C\sum_{j=1}^Id_j^{-1+\frac{d}{2}}\|\Gamma-\Gamma_h\|_{L^2(Q_j^1(z))}.\label{max-stability-proof:6}
\end{align}
Obviously, the remaining task is to approximate $\|\Gamma-\Gamma_h\|_{L^2(Q_j^1(z))}$. To this end, we consider a backward parabolic problem in the following
\begin{subequations}\label{bpp}
\begin{align}
\boldsymbol{\mathcal{W}}+\nabla w=0\quad &\text{in}~\Omega\times [0,T),\label{bpp:1}\\
-\partial_t w+\nabla\cdot\boldsymbol{\mathcal{W}}=v\quad &\text{in}~\Omega\times [0,T),\label{bpp:2}\\
w=0\quad &\text{on}~\partial\Omega\times [0,T],\\
w(x,T)=0\quad &\text{in}~\Omega,
\end{align}
\end{subequations}
where $v$ is a function supported in $Q_j^1(z)$ and $\|v\|_{L^2(Q_j^1(z))}=1$. By Theorem \ref{regularity}, we know that the above problem has a unique solution such that $w\in L^2((0,T);H^{1+\gamma}(\Omega))$. Multiply (\ref{bpp:2}) with $\Gamma-\Gamma_h$ to infer that
\begin{align}
\iint_{Q_T}v(\Gamma-\Gamma_h) dx dt
=&\int_0^T(w,\partial_t(\Gamma-\Gamma_h))_{\mathcal{T}_h} dt+((\Gamma-\Gamma_h)(0),w(0))\nonumber\\
&-\int_0^T(\boldsymbol{\mathcal{W}},\nabla(\Gamma-\Gamma_h))_{\mathcal{T}_h} dt-\int_0^T\langle\Gamma_h-\widehat{\Gamma}_h,\boldsymbol{\mathcal{W}}\cdot\Bn \rangle_{\partial\mathcal{T}_h}dt\label{max-stability-proof:7}\\
=&(\tilde{\delta}_z,w(0)-\Pi_V w(0))+\int_0^T(\partial_t(\Gamma-\Gamma_h),w-w_h)_{\mathcal{T}_h} dt\nonumber\\
&-\int_0^T(\boldsymbol{\mathcal{W}}-\boldsymbol{\mathcal{W}}_h, \nabla(\Gamma-\Gamma_h))_{\mathcal{T}_h} dt-\int_0^T\langle\Gamma_h-\widehat{\Gamma}_h, (\boldsymbol{\mathcal{W}}-\boldsymbol{\mathcal{W}}_h)\cdot\Bn\rangle_{\partial\mathcal{T}_h} dt+\mathcal{K},\nonumber
\end{align}
where $(w_h,\boldsymbol{\mathcal{W}}_h)\in V_h\times \BW_h$ and
\begin{align*}
\mathcal{K}=\int_0^T(\partial_t(\Gamma-\Gamma_h),w_h)_{\mathcal{T}_h} dt-\int_0^T(\boldsymbol{\mathcal{W}}_h, \nabla(\Gamma-\Gamma_h))_{\mathcal{T}_h} dt-\int_0^T\langle\Gamma_h-\widehat{\Gamma}_h, \boldsymbol{\mathcal{W}}_h\cdot\Bn\rangle_{\partial\mathcal{T}_h} dt.
\end{align*}
If $w_h\in V_h\cap H_0^1(\Omega)$, by (\ref{rGreen}), (\ref{rGreenHDG}), (\ref{bpp}) and integration by parts we have
\begin{align}
\mathcal{K}=&\int_0^T(\boldsymbol{\mathcal{R}}-\boldsymbol{\mathcal{R}}_h,\nabla w_h)_{\mathcal{T}_h}-\int_0^T(\boldsymbol{\mathcal{W}}_h,-\boldsymbol{\mathcal{R}}-\nabla\Gamma_h)_{\mathcal{T}_h}+\langle \boldsymbol{\mathcal{W}}_h\cdot\Bn,
\Gamma_h-\widehat{\Gamma}_h\rangle_{\partial\mathcal{T}_h} dt\label{max-stability-proof:ad1}\\
=&\int_0^T(\boldsymbol{\mathcal{R}}-\boldsymbol{\mathcal{R}}_h,\nabla (w_h-w))_{\mathcal{T}_h}-(\boldsymbol{\mathcal{W}}-\boldsymbol{\mathcal{W}}_h,\boldsymbol{\mathcal{R}}-\boldsymbol{\mathcal{R}}_h)_{\mathcal{T}_h} dt.\nonumber
\end{align}
On the other hand, for each $K\in\mathcal{T}_h$ we have
\begin{align*}
h_K^{-1}\|\Gamma_h-\widehat{\Gamma}_h\|^2_{L^2(\partial K)}=&h_K^{-1}\langle e_{\widehat{\Gamma}}-e_{\Gamma}, \Gamma_h-\widehat{\Gamma}_h\rangle_{\partial K}
+h_K^{-1}\langle \Pi_V\Gamma-\Pi_{\widehat{V}}\Gamma, \Gamma_h-\widehat{\Gamma}_h\rangle_{\partial K}\\
\leq &Ch_K^{-1/2}\|\Gamma_h-\widehat{\Gamma}_h\|_{L^2(\partial K)}\Big(h_K^{-1/2}\|e_{\Gamma}-e_{\widehat{\Gamma}}\|_{L^2(\partial K)}\\
&+h_K^{-1/2}\big(\|\Pi_V\Gamma-\Gamma\|_{L^2(\partial K)}+\|\Gamma-P_h^k\Gamma\|_{L^2(\partial K)}\big)\Big),
\end{align*}
where $P_h^k$ denotes the $L^2$-orthogonal projection onto $\mathcal{P}^k(K)$ on each $K\in\mathcal{T}_h$.
This inequality, together with Lemma \ref{tra}, Lemma \ref{H1} and interpolation error estimates, leads to
\begin{align}
h_K^{-1/2}\|\Gamma_h-\widehat{\Gamma}_h\|_{L^2(\partial K)}\leq & C\big(\|e_{\boldsymbol{\mathcal{R}}}\|_{L^2(K)}+\|\sqrt{\tau}(\mathcal{L}e_{\Gamma}-e_{\widehat{\Gamma}})\|_{L^2(\partial K)}+h^{\gamma}|\Gamma|_{H^{1+\gamma}(K)}\big)\nonumber\\
\leq &C\big(\|\boldsymbol{\mathcal{R}}-\boldsymbol{\mathcal{R}}_h\|_{L^2(K)}+\|\sqrt{\tau}(\mathcal{L}\Gamma_h-\widehat{\Gamma}_h)\|_{L^2(\partial K)}+h^{\gamma}|\Gamma|_{H^{1+\gamma}(K)}\big).\label{max-stability-proof:ad2}
\end{align}
Then by setting $w_h=I_hw$ and $\boldsymbol{\mathcal{W}}_h=\Pi_W\boldsymbol{\mathcal{W}}$ and using (\ref{max-stability-proof:7}), (\ref{max-stability-proof:ad1}), (\ref{max-stability-proof:ad2}), Lemma \ref{tra} and interpolation error estimates, we have
\begin{align}
\|\Gamma-\Gamma_h\|_{L^2(Q_j^1(z))}\leq &Ch^{1+\gamma-\frac{d}{2}}|w(0)|_{H^{1+\gamma}(\Omega\backslash\Omega_j^2(z))}+C\sum_{i\in\{\ast,1,\cdots,I\}}
|w|_{L^2H^{1+\gamma}(Q_i^1(z))}\times\nonumber\\
&\Big(h^{1+\gamma}\|\partial_t(\Gamma-\Gamma_h)\|_{L^2(Q_i(z))}+h^{\gamma}\|\nabla(\Gamma-\Gamma_h)\|_{L^2(Q_i(z))}
+h^{2\gamma}|\Gamma|_{L^2H^{1+\gamma}(Q_i^1(z))}\label{max-stability-proof:9}\\
&+h^{\gamma}\left(\|\boldsymbol{\mathcal{R}}-\boldsymbol{\mathcal{R}}_h\|_{L^2(Q_i^1(z))}
+\interleave(\Gamma_h,\widehat{\Gamma}_h)\interleave_{L^2(Q_i^1(z))}\right)\Big),\nonumber
\end{align}
where $I_h$ denotes the Lagrange interpolation operator.
From \cite[(5.24) and (5.31)]{l2019}, we know that
\begin{align}
|w(0)|_{H^{1+\gamma}(\Omega\backslash\Omega_j^2(z))}\leq Cd_j^{-\gamma},\quad |w|_{L^2H^{1+\gamma}(Q_i^1(z))}\leq Cd_i^{1-\gamma}\left(\frac{\min\{d_i,d_j\}}{\max\{d_i,d_j\}}\right)^{\gamma}=Cd_i^{1-\gamma}m_{ij}^{\gamma}.\label{max-stability-proof:ad5}
\end{align}
Thus we have
\begin{align}
\sum_{j=1}^Id_j^{-1+\frac{d}{2}}\|\Gamma-\Gamma_h\|_{L^2(Q_j^1(z))}&
\leq C\sum_{j=1}^I(hd_j^{-1})^{1+\gamma-\frac{d}{2}}
+C\sum_{i\in\{\ast,1,\cdots,I\}}\Big(h^{1+\gamma}\|\partial_t(\Gamma-\Gamma_h)\|_{L^2(Q_i(z))}\nonumber\\
&+h^{\gamma}\Big(\|\nabla(\Gamma-\Gamma_h)\|_{L^2(Q_i(z))}+\|\boldsymbol{\mathcal{R}}-\boldsymbol{\mathcal{R}}_h\|_{L^2(Q_i(z))}
+\interleave(\Gamma_h,\widehat{\Gamma}_h)\interleave_{L^2(Q_i(z))}\Big)\label{max-stability-proof:ad7}\\
&+h^{2\gamma}d_i^{-\gamma-\frac{d}{2}}\Big)\sum_{j=1}^Id_j^{-1+\frac{d}{2}}d_{i}^{1-\gamma}m_{ij}^{\gamma},\nonumber
\end{align}
which, together with
\begin{align*}
\sum_{j=1}^Id_j^{-1+\frac{d}{2}}d_{i}^{1-\gamma}m_{ij}^{\gamma}
=d_i^{\frac{d}{2}-\gamma}\left(\sum_{j=1}^i\left(\frac{d_i}{d_j}\right)^{\gamma+1-\frac{d}{2}}+\sum_{j=i+1}^I\left(\frac{d_j}{d_i}
\right)^{\gamma+\frac{d}{2}-1}\right)\leq Cd_i^{\frac{d}{2}-\gamma},
\end{align*}
leads to
\begin{align*}
\sum_{j=1}^Id_j^{1+\frac{d}{2}}\mathcal{M}_j\leq& C+C\sum_{i\in\{\ast,1,\cdots,I\}}\left(hd_i^{-1}\right)^{2\gamma}
+C\sum_{i\in\{\ast,1,\cdot,I\}}d_i^{1+\frac{d}{2}}\mathcal{M}_i\left(hd_i^{-1}\right)^{\gamma}.
\end{align*}
By choosing $M$ to be large enough, we have
\begin{align}
\sum_{j=1}^Id_j^{1+\frac{d}{2}}\mathcal{M}_j\leq C.\label{max-stability-proof:10}
\end{align}
Then we conclude the proof by applying the above inequality to (\ref{max-stability-proof:1}).

{\emph{Proof of (\ref{max-stability:1})}:
By Lemma \ref{prep}, (\ref{max-stability-proof:-1}), (\ref{max-stability-proof:0}) and (\ref{max-stability-proof:01}), we have
\begin{align*}
u_h(z,t)\leq& \|f\|_{L^{\infty}((0,t);L^{\infty}(\Omega))}(\|\Gamma_h-\Gamma\|_{L^1((0,t);L^1(\Omega))}+\|\Gamma\|_{L^1((0,t):L^1(\Omega))})\\
&+\|u_{0,h}\|_{L^{\infty}(\Omega)}(\|\partial_s(\Gamma_h-\Gamma)\|_{L^1((0,t);L^1(\Omega))}+\|\Gamma\|_{L^1(\Omega)})\\
\leq &\|f\|_{L^{\infty}((0,t);L^{\infty}(\Omega))}(\|\Gamma_h-\Gamma\|_{L^1((0,t);L^1(\Omega))}+\|\Gamma\|_{L^1((0,t):L^1(\Omega))})
+C\|u_{0,h}\|_{L^{\infty}(\Omega)}.
\end{align*}
In order to complete the proof, it is obvious from the above inequality that we only need to prove the following estimate:
\begin{align*}
\|\Gamma_h-\Gamma\|_{L^1((0,t);L^1(\Omega))}+\|\Gamma\|_{L^1((0,t):L^1(\Omega))}\leq C.
\end{align*}

By (\ref{de}), Lemma \ref{ee}, (\ref{ee-proof:ad4}) and Poincar\'{e} inequality, we obtain that
\begin{align*}
&\|\Gamma_h-\Gamma\|_{L^1((0,T);L^1(\Omega))}+\|\Gamma\|_{L^1((0,T):L^1(\Omega))}\\
\leq& C(Mh)^{1+\frac{d}{2}}\left(\|\Gamma_h-\Gamma\|_{L^2(Q_{\ast}(z))}+\|\Gamma\|_{L^2(Q_{\ast}(z))}\right)\\
&+C\sum_{j=1}^Id_j^{1+\frac{d}{2}}\left(\|\Gamma_h-\Gamma\|_{L^2(Q_{j}(z))}+\|\Gamma\|_{L^2(Q_j(z))}\right)\\
\leq& C(Mh)^{1+\frac{d}{2}}\Big(\|\nabla\Gamma_h\|_{L^2((0,T);L_h^2(\Omega))}+\Big(\int_0^T\sum_{K\in\mathcal{T}_h}h_K^{-1}\|\Gamma_h-\widehat{\Gamma}_h\|^2_{L^2(\partial K)} dt\Big)^{1/2}\\
&+\|\nabla\Gamma\|_{L^2(Q_T)}\Big)+C\sum_{j=1}^Id_j^{1+\frac{d}{2}}\left(\|\Gamma_h-\Gamma\|_{L^2(Q_{j}(z))}+\|\Gamma\|_{L^2(Q_j(z))}\right)\\
\leq& Ch+C\sum_{j=1}^Id_j^{1+\frac{d}{2}}\left(\|\Gamma_h-\Gamma\|_{L^2(Q_{j}(z))}+\|\Gamma\|_{L^2(Q_j(z))}\right).
\end{align*}
Note that in the derivation of the above inequality we have used the following discrete Sobolev embedding inequality (see \cite[Theorem 2.1]{de2010})
for $1\leq p\leq 6$ and any $\mu\in L_0^2(\mathcal{E}_h)=\{v: v|_F\in L^2(F)~\forall F\in\mathcal{E}_h,~v=0~\text{on}~\mathcal{E}_h^{\partial}\}$:
\begin{align}
\|\Gamma_h\|_{L^p(\Omega)}\leq C\Big(\|\nabla\Gamma_h\|^2_{L_h^2(\Omega)}+\sum_{K\in\mathcal{T}_h}h_K^{-1}\|\Gamma_h-\mu\|^2_{L^2(\partial K)}\Big)^{1/2}.\label{max-stability-proof:ad11}
\end{align}
On the other hand, with (\ref{max-stability-proof:ad7}) and (\ref{max-stability-proof:10}) we have
\begin{align}
&\sum_{j=1}^Id_j^{1+\frac{d}{2}}\left(\|\Gamma_h-\Gamma\|_{L^2(Q_{j}(z))}+\|\Gamma\|_{L^2(Q_j(z))}\right)\nonumber\\
\leq& \sum_{j=1}^{I}d_j^{-1+\frac{d}{2}}\|\Gamma_h-\Gamma\|_{L^2(Q^1_{j}(z))}+\sum_{j=1}^Id_j^{1+\frac{d}{2}}\|\Gamma\|_{L^2(Q_j(z))}\label{max-stability-proof:11}\\
\leq& C+C\sum_{i\in\{\ast,1,\cdots,I\}}d_i^{1+\frac{d}{2}}\mathcal{M}_i\left(hd_i^{-1}\right)^{\gamma}+\sum_{j=1}^Id_j^{1+\frac{d}{2}}\|\Gamma\|_{L^2(Q_j(z))}\nonumber\\
\leq& C+\sum_{j=1}^Id_j^{1+\frac{d}{2}}\|\Gamma\|_{L^2(Q_j(z))},\nonumber
\end{align}
if $M$ is chosen to be large enough. Thus the remaining task is to estimate $\|\Gamma\|_{L^2(Q_j(z))}$. To this end, let $(\boldsymbol{\mathcal{W}},w)$ be the solution of problem (\ref{bpp}) with $v$ being a function supported in $Q_j(z)$ and $\|v\|_{L^2(Q_j(z))}= 1$. Then by (\ref{rddf}), (\ref{rGreen}), (\ref{bpp}) and a simple calculation, we have
\begin{align*}
\iint_{Q_T}v\Gamma dxdt=&(\widetilde{\delta}_z(x),w(x,0))+\int_0^T(w,\partial_t\Gamma) dt-\int_0^T(\boldsymbol{\mathcal{W}},\nabla\Gamma) dt\\
=&(\widetilde{\delta}_z(x),w(x,0))+\int_0^T(\boldsymbol{\mathcal{R}},\nabla w)dt-\int_0^T(\boldsymbol{\mathcal{W}},\nabla\Gamma) dt\\
=&(\widetilde{\delta}_z(x),w(x,0))\leq C\|w(x,0)\|_{L^{\infty}(K_z)}.
\end{align*}
Using the Green's function defined in (\ref{Green}) and (\ref{bpp}), we have the following expression:
\begin{align*}
w(x,0)=&(w(y,0),\delta_x(y))=(w(y,0),G(x,y,0))\nonumber\\
=&-\int_{\Omega}\int_0^T\partial_t(w(y,t)G(x,y,t)) dt dy\\
=&\iint_{Q_T}v(y,t)G(x,y,t)dydt=\iint_{Q_j(z)}vG(x,y,t) dydt,\nonumber
\end{align*}
which, together with (\ref{GE}), leads to
\begin{align*}
|w(x,0)|\leq \|v\|_{L^2(Q_j(z))}\|G(x,\cdot,\cdot)\|_{L^2(Q_j^1(x))}\leq Cd_j^{1-\frac{d}{2}},\quad \forall x\in K_z.
\end{align*}
Therefore
\begin{align*}
\|\Gamma\|_{L^2(Q_j(z))}\leq Cd_j^{1-\frac{d}{2}},
\end{align*}
which, together with (\ref{max-stability-proof:11}), yields that
\begin{align*}
\sum_{j=1}^Id_j^{1+\frac{d}{2}}\left(\|\Gamma_h-\Gamma\|_{L^2(Q_{j}(z))}+\|\Gamma\|_{L^2(Q_j(z))}\right)\leq C+C\sum_{j=1}^Id_j^2\leq C.
\end{align*}
Then we complete the proof.

\section{Maximum-norm error estimates}\label{sec5}
We define the discrete Laplacian operator $\Delta_h$ by
\begin{align}
(\Delta_h u_h,v_h)=(\Bq_h,\nabla v_h)_{\mathcal{T}_h}-\langle\Bq_h\cdot\textbf{n}+\tau(\mathcal{L}u_h-\widehat{u}_h),v_h\rangle_{\partial\mathcal{T}_h},\quad \forall u_h,v_h\in V_h,\label{dl:1}
\end{align}
where $(\Bq_h,\widehat{u}_h)\in \BW_h\times \widehat{V}_h$ is the solution of the following problem
\begin{align}
(\Bq_h,\Bw_h)_{\mathcal{T}_h}-(u_h,\nabla\cdot\Bw_h)_{\mathcal{T}_h}+\langle\widehat{u}_h,\Bw_h\cdot\textbf{n}\rangle_{\partial \mathcal{T}_h}=0,&\quad \forall \Bw_h\in\BW_h,\label{dl:2}\\
\langle\Bq_h\cdot\textbf{n}+\tau(\mathcal{L}u_h-\widehat{u}_h),\widehat{v}_h\rangle_{\partial\mathcal{T}_h}=0,&\quad \forall \widehat{v}_h\in\widehat{V}_h.\label{dl:3}
\end{align}
By Lemma \ref{ee1}, it is easy to observe that $(\Bq_h,\widehat{u}_h)$ is uniquely solvable for the given $u_h$, thus the discrete Laplacian operator is well-defined and the semi-discrete scheme (\ref{HDG}) can be rewritten as
\begin{align}
\partial_t u_h-\Delta_h u_h=f_h=\Pi_V f.\label{dl:4}
\end{align}

The main results of this section are to prove Theorem~\ref{stability} and Corollary~\ref{max-ee}.
\begin{theorem}\label{stability}
Let $(\Bq_h,u_h,\widehat{u}_h)\in\BW_h\times V_h\times \widehat{V}_h$ be the solution of problem (\ref{HDG}) and let the discrete Laplacian operator $\Delta_h$
be defined as above, then if $u_{0,h}=0$ we have
\begin{align}
\|\partial_t u_h\|_{L^{\infty}((0,T);L^{p}(\Omega))}+\|\Delta_h u_h\|_{L^{\infty}((0,T);L^{p}(\Omega))}\leq C|\ln h|\|f\|_{L^{\infty}((0,T);L^{p}(\Omega))},\quad \forall 1\leq p\leq \infty,
\end{align}
where $C$ is a positive constant independent of $T$.
\end{theorem}
\begin{proof}
By Lemma \ref{prep}, we have for any $z\in K_z$ with $K_z\in\mathcal{T}_h$ and $t>0$
\begin{align}
\partial_t u_h(z,t)=\int_0^t\int_{\Omega}f_h(x,s)\partial_t\Gamma_h(z,x,t-s)dxds+f_h(z,t).\nonumber
\end{align}
Therefore for any $1\leq p\leq \infty$, by using a similar procedure with that presented in \cite[Subsection 4.5]{l2019} we obtain
\begin{align}
\|\partial_t u_h(\cdot,t)\|_{L^p(\Omega)}\leq& C\left\|\int_0^t\int_{\Omega}f_h(x,s)\partial_t\Gamma_h(\cdot,x,t-s)dxds\right\|_{L^p(\Omega)}+
\|f_h(\cdot,t)\|_{L^p(\Omega)}\nonumber\\
\leq&C\left(\int_0^t\sup_{z\in\Omega}\int_{\Omega}
\left|\partial_s\Gamma_h(z,x,s)\right|dxds\right)\|f_h\|_{L^{\infty}((0,t);L^p(\Omega))}+\|f_h(\cdot,t)\|_{L^p(\Omega)}.\nonumber
\end{align}
Obviously, the proof will be finished if we can prove
\begin{align}
\int_0^t\sup_{z\in\Omega}\int_{\Omega}
\left|\partial_s\Gamma_h(z,x,s)\right|dxds\leq C|\ln h|.\label{stablity-proof:1}
\end{align}

By (\ref{max-stability-proof:0}), (\ref{max-stability-proof:01}) and triangle inequality, we have
\begin{align}
\|t\partial_t\Gamma_h(z,\cdot,t)\|_{L^1(\Omega)}\leq& \|t\partial_t(\Gamma_h-\Gamma)\|_{L^1(\Omega)}+\|t\partial_t\Gamma\|_{L^1(\Omega)}\nonumber\\
\leq&C+\left\|\int_0^ts\partial_{ss}(\Gamma_h-\Gamma)+\partial_s(\Gamma_h-\Gamma) ds\right\|_{L^1(\Omega)}\label{stability-proof:2}\\
\leq& C+\|s\partial_{ss}(\Gamma_h-\Gamma)\|_{L^1((0,t);L^1(\Omega))}+\|\partial_s(\Gamma_h-\Gamma)\|_{L^1((0,t);L^1(\Omega))}\leq C.\nonumber
\end{align}
On the other hand, by finding the derivative of (\ref{rGreenHDG}) and choosing the test function as $(\Bw_h,v_h,\widehat{v}_h)=(\partial_t\boldsymbol{\mathcal{R}}_h,\partial_t\Gamma_h,\partial_t\widehat{\Gamma}_h)$, we obtain
\begin{align*}
\frac{1}{2}\frac{d}{dt}\|\partial_t\Gamma_h\|_{L^2(\Omega)}^2+\|\partial_t\boldsymbol{\mathcal{R}}_h\|_{L^2(\Omega)}^2+
\sum_{K\in\mathcal{T}_h}\|\sqrt{\tau}\partial_t(\mathcal{L}\Gamma_h-\widehat{\Gamma}_h)\|^2_{L^2(\partial K)}=0,
\end{align*}
which, together with (\ref{ee-proof:ad4}) and the discrete Sobolev embedding inequality (\ref{max-stability-proof:ad11}), leads to
\begin{align*}
\frac{d}{dt}\|\partial_t\Gamma_h\|_{L^2(\Omega)}^2+C\|\partial_t\Gamma_h\|_{L^2(\Omega)}^2\leq 0,
\end{align*}
where $C$ is a positive constant independent of $h$ and the time $t$. From the above inequality, we derive the exponential decay of $\partial_t\Gamma_h$
with respect to $t\geq t_0$
\begin{align}
\|\partial_t\Gamma_h(z,\cdot,t)\|^2_{L^2(\Omega)}\leq& e^{-C(t-t_0)}\|\partial_t\Gamma_h(z,\cdot,t_0)\|_{L^2(\Omega)}^2\nonumber\\
\leq & e^{-C(t-t_0)}e^{-Ct_0}\|\partial_t\Gamma_h(z,\cdot,0)\|_{L^2(\Omega)}^2\label{stability-proof:3}\\
\leq & e^{-C(t-t_0)}e^{-Ct_0}h^{-4-d}\leq e^{-C(t-t_0)},\nonumber
\end{align}
by setting $t_0=\frac{(4+d)}{C}\ln\frac{1}{h}$ and using (\ref{ee-proof:4}). Moreover, with (\ref{stability-proof:2}) and (\ref{ee-proof:4}) we can obtain
for any $\theta\in (0,1)$ and $t\in (0,t_0)$
\begin{align*}
\|\partial_t\Gamma_h(z,\cdot,t)\|_{L^1(\Omega)}=&\|\partial_t\Gamma_h(z,\cdot,t)\|_{L^1(\Omega)}^{1-\theta}\|\partial_t\Gamma_h(z,\cdot,t)\|_{L^1(\Omega)}^{\theta}\nonumber\\
\leq&Ct^{\theta-1}\|\partial_t\Gamma_h(z,\cdot,0)\|_{L^2(\Omega)}^{\theta}=\frac{C}{h^{(2+\frac{d}{2})\theta}t^{1-\theta}},
\end{align*}
which, together with (\ref{stability-proof:3}), results in
\begin{align*}
\int_0^t\sup_{z\in\Omega}\int_{\Omega}\left|\partial_s\Gamma_h(z,x,s)\right|dxds=&\int_0^{t_0}\sup_{z\in\Omega}\int_{\Omega}
\left|\partial_s\Gamma_h(z,x,s)\right|dxds+\int_{t_0}^t\sup_{z\in\Omega}\int_{\Omega}
\left|\partial_s\Gamma_h(z,x,s)\right|dxds\\
\leq &\int_0^{t_0}\frac{C}{h^{(2+\frac{d}{2})\theta}t^{1-\theta}}dt+\int_{t_0}^tCe^{-C(t-t_0)}dt\\
=&\frac{Ct_0^{\theta}}{\theta h^{(2+\frac{d}{2})\theta}}+C\leq C\ln\frac{1}{h},
\end{align*}
by choosing $\theta=\frac{1}{(2+\frac{d}{2})\ln\frac{1}{h}}$. Thus,
$$\|\partial_tu_h(\cdot,t)\|_{L^p(\Omega)}\leq C|\ln h|\|f_h\|_{L^{\infty}((0,T);L^p(\Omega))},\quad \forall 0<t<T,$$
which, together with (\ref{dl:4}), completes the proof.
\end{proof}

Let $(\Bq_h^e,u_h^e,\widehat{u}_h^e)\in \BW_h\times V_h\times \widehat{V}_h$ be the solution of the following problem:
\begin{subequations}\label{HDG-E}
\begin{align}
(\Bq_h^e,\Bw_h)_{\mathcal{T}_h}-(u_h^e,\nabla\cdot\Bw_h)_{\mathcal{T}_h}+\langle\widehat{u}_h^e,\Bw_h\cdot\textbf{n}\rangle_{\partial \mathcal{T}_h}=0,\quad \forall \Bw_h\in\BW_h\\
-(\Bq_h^e,\nabla v_h)_{\mathcal{T}_h}+\langle\Bq_h^e\cdot\textbf{n}+\tau(\mathcal{L}u_h^e-\widehat{u}_h^e),v_h\rangle_{\partial\mathcal{T}_h}=(f-\partial_tu,v_h),\quad \forall v_h\in V_h,\\
\langle\Bq_h^e\cdot\textbf{n}+\tau(\mathcal{L}u_h^e-\widehat{u}_h^e),\widehat{v}_h\rangle_{\partial\mathcal{T}_h}=0,\quad \forall \widehat{v}_h\in\widehat{V}_h.
\end{align}
\end{subequations}
Obviously, (\ref{HDG-E}) is the elliptic version of (\ref{HDG}) for the Poisson equations with external force $f-\partial_t u$. Then based on Theorem \ref{stability}, we have the following maximum-norm error estimate.
\begin{corollary}\label{max-ee}
Let $u$, $(\Bq_h,u_h,\widehat{u}_h)$ and $(\Bq_h^e,u_h^e,\widehat{u}_h^e)$ be the solutions of problems (\ref{model}), (\ref{HDG}) and (\ref{HDG-E})
respectively. Let the discrete Laplacian operator $\Delta_h$ be defined as above. Then we have
\begin{align}
\|u-u_h\|_{L^{\infty}((0,T);L^{\infty}(\Omega))}\leq& C|\ln h|\Big(\|u-u_h^e\|_{L^{\infty}((0,T);L^{\infty}(\Omega))}+\|u-\Pi_Vu\|_{L^{\infty}((0,T);L^{\infty}(\Omega))}\nonumber\\
&+\|\Pi_Vu(0)-u_h(0)\|_{L^{\infty}(\Omega)}\Big).
\end{align}
\end{corollary}
\begin{proof}
Let $w_h=\Pi_Vu-u_h-e^{-t}(\Pi_Vu(0)-u_h(0))$, according to the definitions of $\Delta_h$ and $u_h^e$ we have
\begin{align*}
\partial_t w_h-\Delta_hw_h=&\Delta_h(u_h^e-\Pi_Vu+e^{-t}(\Pi_Vu(0)-u_h(0)))+e^{-t}(\Pi_Vu(0)-u_h(0))\\
&+\partial_t(\Pi_Vu-u_h)-\Delta_h(u_h^e-u_h)\\
=&\Delta_h(u_h^e-\Pi_Vu+e^{-t}(\Pi_Vu(0)-u_h(0)))+e^{-t}(\Pi_Vu(0)-u_h(0))
\end{align*}
with $w_h(0)=0$. Since the discrete Laplacian operator $\Delta_h$ is well-defined, we can multiply the above equation with $\Delta_h^{-1}$ and obtain
\begin{align*}
\partial_t(\Delta_h^{-1}w_h)-\Delta_h(\Delta_h^{-1}w_h)=u_h^e-\Pi_Vu+(e^{-t}+e^{-t}\Delta_h^{-1})(\Pi_Vu(0)-u_h(0)),
\end{align*}
with $\Delta_h^{-1}w_h(0)=0$. Then according to Theorem \ref{stability} we get
\begin{align}
\|w_h\|_{L^{\infty}((0,T);L^{\infty}(\Omega))}=&\|\Delta_h(\Delta_h^{-1}w_h)\|_{L^{\infty}((0,T);L^{\infty}(\Omega))}\nonumber\\
\leq& C|\ln h|\Big(\|u_h^e-\Pi_Vu\|_{L^{\infty}((0,T);L^{\infty}(\Omega))}+\|\Pi_Vu(0)-u_h(0)\|_{L^{\infty}(\Omega)}\label{stability-proof:4}\\
&+\|\Delta_h^{-1}(\Pi_Vu(0)-u_h(0))\|_{L^{\infty}(\Omega)}\Big)\nonumber\\
\leq& C|\ln h|\Big(\|u_h^e-\Pi_Vu\|_{L^{\infty}((0,T);L^{\infty}(\Omega))}+\|\Pi_Vu(0)-u_h(0)\|_{L^{\infty}(\Omega)}\Big).\nonumber
\end{align}
Note that in the derivation of the above inequality, the following result, that will be proved in the Appendix A, has been used
\begin{align*}
\|\Delta_h^{-1}(\Pi_Vu(0)-u_h(0))\|_{L^{\infty}(\Omega)}\leq C\|\Pi_Vu(0)-u_h(0)\|_{L^{\infty}(\Omega)}.
\end{align*}
Finally, by using the stability of the $L^2$-projection $\Pi_V$ to (\ref{stability-proof:4}), we conclude the proof.
\end{proof}

\section{Conclusions}\label{sec6}
In this paper, for three different HDG methods (including mixed methods), we provide a unified semi-discrete scheme for parabolic problem in nonconvex polygonal/polyhedral domains. By establishing the local energy error estimates and energy estimates of the regularized Green's function, we provide a unified
framework to prove the maximum-norm stabilities of both the semigroup defined by the semi-discrete scheme and the corresponding discrete solutions. Since the asymmetry of the discrete scheme induced by the introduction of flux variables, the proof of stabilities in this paper is more difficult than our previous paper \cite{clq2025}. By using these stability results, we further prove the maximal regularity in $L^{\infty}((0,T);L^p(\Omega))$-norm and maximum-norm error estimates by introducing the discrete Laplacian operator associated with HDG methods. Since the proof of the local energy error estimates has not used the lifting operator that is only available for simplicial meshes (see \cite{lc2023}), our results (excluding mixed methods) are also valid for polygonal/polyhedral meshes.

\smallskip
\smallskip

\appendix

\section{}

Let $\Delta_h^{-1}\phi_h=\psi_h$ for any $\phi_h,\psi_h\in V_h$, we have $\Delta_h\psi_h=\phi_h$ and it suffices to prove that
\begin{align*}
\|\psi_h\|_{L^{\infty}(\Omega)}\leq C\|\phi_h\|_{L^{\infty}(\Omega)}.
\end{align*}
To this end, let $(\widetilde{\Bq},\psi)$ be the solution of the following problem:
\begin{subequations}
\begin{align}
\widetilde{\Bq}+\nabla\psi=0,\quad \text{in}~\Omega,\\
\nabla\cdot\widetilde{\Bq}=-\phi_h,\quad \text{in}~\Omega,\\
\psi=0,\quad \text{on}~\partial\Omega.
\end{align}
\end{subequations}
According to \cite[Theorem 18.13]{d1988}, we know that the above problem has a unique solution $\psi\in H_0^1(\Omega)\cap H^{1+\gamma}(\Omega)$
with $\gamma\in (\frac{1}{2},1)$ satisfying $\|\psi\|_{H^{1+\gamma}(\Omega)}\leq C\|\phi_h\|_{L^2(\Omega)}$. On the other hand, let $(\widetilde{\Bq}_h,\widehat{\psi}_h)\in \BW_h\times\widehat{V}_h$ be the solution of
(\ref{dl:2})-(\ref{dl:3}) with $u_h$ replaced by $\psi_h$, then we have
\begin{align}
a_h(\widetilde{\Bq}_h,\psi_h,\widehat{\psi}_h;\Bw_h,v_h,\widehat{v}_h)=-(\phi_h,v_h),\quad \forall (\Bw_h,v_h,\widehat{v}_h)\in\BW_h\times V_h\times\widehat{V}_h,\label{apendixA:1}
\end{align}
by using the definition of $\Delta_h$, where
\begin{align*}
&a_h(\widetilde{\Bq}_h,\psi_h,\widehat{\psi}_h;\Bw_h,v_h,\widehat{v}_h)\\
=&(\widetilde{\Bq}_h,\Bw_h)_{\mathcal{T}_h}+(\nabla \psi_h,\Bw_h)_{\mathcal{T}_h}-(\widetilde{\Bq}_h,\nabla v_h)_{\mathcal{T}_h}
-\langle\psi_h-\widehat{\psi}_h,\Bw_h\cdot\textbf{n}\rangle_{\partial\mathcal{T}_h}\\
&+\langle\widetilde{\Bq}_h\cdot\textbf{n}+
\tau(\mathcal{L}\psi_h-\widehat{\psi}_h),v_h-\widehat{v}_h\rangle_{\partial\mathcal{T}_h}.
\end{align*}
With a simple calculation, it is easy to observe that
\begin{align}
\|\Bw_h\|^2_{L^2(\Omega)}+\sum_{K\in\mathcal{T}_h}
\|\sqrt{\tau}(\mathcal{L}v_h-\widehat{v}_h)\|^2_{L^2(\partial K)}
= a_h(\Bw_h,v_h,\widehat{v}_h;\Bw_h,v_h,\widehat{v}_h), \label{apendixA:2}
\end{align}
and
\begin{align}
a_h(\widetilde{\Bq},\psi,\psi;\Bw_h,v_h,\widehat{v}_h)=-(\phi_h,v_h)+\langle\tau(\mathcal{L}\psi-\psi),v_h-\widehat{v}_h \rangle_{\partial\mathcal{T}_h},\label{apendixA:3}
\end{align}
for any $(\Bw_h,v_h,\widehat{v}_h)\in\BW_h\times V_h\times\widehat{V}_h$.
Moreover for any $(\Bw,v,\mu)\in H_h^{\gamma}(\Omega)\times H_h^1(\Omega)\times L_0^2(\mathcal{E}_h)$ and
$(\Bw_h,v_h,\widehat{v}_h)\in \BW_h\times V_h\times \widehat{V}_h$, we have by Lemma \ref{inv} and Lemma \ref{tra}
\begin{align}
&a_h(\Bw,v,\mu;\Bw_h,v_h,\widehat{v}_h)\nonumber\\
\leq& \left(\|\Bw\|_{L^2(\Omega)}+\|\nabla v\|_{L_h^2(\Omega)}\right)\|\Bw_h\|_{L^2(\Omega)}+\|\Bw\|_{L^2(\Omega)}\|\nabla v_h\|_{L_h^2(\Omega)}\nonumber\\
&+C\sum_{K\in\mathcal{T}_h}\left(\|\Bw\|_{L^2(K)}+h_K^{\gamma}|\Bw|_{H^{\gamma}(K)}+\|\sqrt{\tau}(\mathcal{L}v-\mu)\|_{L^2(\partial K)}\right)
h_K^{-\frac{1}{2}}\|v_h-\widehat{v}_h\|_{L^2(\partial K)}\nonumber\\
&+C\sum_{K\in\mathcal{T}_h}h_K^{-\frac{1}{2}}\|v-\mu\|_{L^2(\partial K)}\|\Bw_h\|_{L^2(K)}\label{apendixA:4}\\
\leq& C\Big(\|\Bw\|_{L^2(\Omega)}+\|\nabla v\|_{L_h^2(\Omega)}+ h^{\gamma}|\Bw|_{H_h^{\gamma}(\Omega)}+\Big(\sum_{K\in\mathcal{T}_h}\Big(\|\sqrt{\tau}(\mathcal{L}v-\mu)\|_{L^2(\partial K)}^2\nonumber\\
&+h_K^{-1}\|v-\mu\|_{L^2(\partial K)}^2\Big)\Big)^{1/2}\Big)\Big(\|\Bw_h\|_{L^2(\Omega)}+\|\nabla v_h\|_{L_h^2(\Omega)}+\Big(\sum_{K\in\mathcal{T}_h}
h_K^{-1}\|v_h-\widehat{v}_h\|_{L^2(\partial K)}^2\Big)^{1/2}\Big),\nonumber
\end{align}
Then Lemma \ref{H1}, the discrete Sobolev embedding inequality (\ref{max-stability-proof:ad11}) and (\ref{apendixA:1})-(\ref{apendixA:3}) yield
\begin{align}
\|\Pi_V\psi-\psi_h\|^2_{L^6(\Omega)}\leq& C\Big(\|\nabla(\Pi_V\psi-\psi_h)\|_{L_h^2(\Omega)}^2+\sum_{K\in\mathcal{T}_h}h_K^{-1}\|
(\Pi_V\psi-\psi_h)-(\Pi_{\widehat{V}}\psi-\widehat{\psi}_h)\|^2_{L^2(\partial K)}\Big)\nonumber\\
\leq&C\Big(\|\nabla(\Pi_V\psi-\psi_h)\|_{L_h^2(\Omega)}^2+\|\Pi_W\widetilde{\Bq}-\widetilde{\Bq}_h\|^2_{L^2(\Omega)}\nonumber\\
&+\sum_{K\in\mathcal{T}_h}h_K^{-1}\|(\Pi_V\psi-\psi_h)-(\Pi_{\widehat{V}}\psi-\widehat{\psi}_h)\|^2_{L^2(\partial K)}\Big)\nonumber\\
\leq& Ca_h(\Pi_W\widetilde{\Bq}-\widetilde{\Bq}_h,\Pi_V\psi-\psi_h,\Pi_{\widehat{V}}\psi-\widehat{\psi}_h;
\Pi_W\widetilde{\Bq}-\widetilde{\Bq}_h,\Pi_V\psi-\psi_h,\Pi_{\widehat{V}}\psi-\widehat{\psi}_h)\nonumber\\
=&Ca_h(\Pi_W\widetilde{\Bq}-\widetilde{\Bq},\Pi_V\psi-\psi,\Pi_{\widehat{V}}\psi-\psi;
\Pi_W\widetilde{\Bq}-\widetilde{\Bq}_h,\Pi_V\psi-\psi_h,\Pi_{\widehat{V}}\psi-\widehat{\psi}_h)\nonumber\\
&+\langle\tau(\mathcal{L}\psi-\psi),(\Pi_V\psi-\psi_h)-(\Pi_{\widehat{V}}\psi-\widehat{\psi}_h) \rangle_{\partial\mathcal{T}_h},\nonumber
\end{align}
which, together with (\ref{apendixA:4}) and the interpolation error estimates, leads to
\begin{align}
\|\Pi_V\psi-\psi_h\|_{L^6(\Omega)}\leq& C\Big(\|\Pi_W\widetilde{\Bq}-\widetilde{\Bq}\|_{L^2(\Omega)}+\|\nabla(\Pi_V\psi-\psi)\|_{L_h^2(\Omega)}+ h^{\gamma}|\Pi_W\widetilde{\Bq}-\widetilde{\Bq}|_{H_h^{\gamma}(\Omega)}\nonumber\\
&+\Big(\sum_{K\in\mathcal{T}_h}h_K^{-1}\big(\|\Pi_V\psi-\psi\|^2_{L^2(\partial K)}+\|\Pi_{\widehat{V}}\psi-\psi\|_{L^2(\partial K)}^2\big)\Big)^{1/2}\Big)\nonumber\\
\leq& Ch^{\gamma}\|\psi\|_{H^{1+\gamma}(\Omega)}\leq Ch^{\gamma}\|\phi_h\|_{L^2(\Omega)}. \nonumber
\end{align}
Thus we have
\begin{align*}
\|\psi_h\|_{L^{\infty}(\Omega)}\leq& \|\psi_h-\Pi_V\psi\|_{L^{\infty}(\Omega)}+\|\Pi_V\psi-\psi\|_{L^{\infty}(\Omega)}+\|\psi\|_{L^{\infty}(\Omega)}\\
\leq &C\Big(h^{-\frac{d}{6}}\|\psi_h-\Pi_V\psi\|_{L^6(\Omega)}+h^{1+\gamma-\frac{d}{2}}\|\psi\|_{H^{1+\gamma}(\Omega)}+\|\phi_h\|_{L^2(\Omega)}\Big)\\
\leq & C\Big(h^{\gamma-\frac{1}{2}}\|\phi_h\|_{L^2(\Omega)}+h^{1+\gamma-\frac{3}{2}}\|\phi_h\|_{L^2(\Omega)}+\|\phi_h\|_{L^2(\Omega)}\Big)\\
\leq & C\|\phi_h\|_{L^{\infty}(\Omega)}.
\end{align*}
Then we conclude the proof.

\smallskip
\smallskip

{\bf Acknowledgements} All authors contributed equally. Huangxin Chen was supported by National Key Research and Development Project of China (Grant No. 2024YFA1012600), National Natural Science Foundation of China (Grant No. 12471345), and Fujian Provincial Natural Science Foundation of China (2026J011002). Haitao Leng was supported by Basic and Applied Basic Research Foundation of Guangdong Province (Grant No. 2024A1515011163, 2025A1515010428). Weifeng Qiu was supported by the Research Grants Council of the Hong Kong Special Administrative
Region, China. (Project No. CityU 11300621). Weifeng Qiu would like to thank Professor Rob Stevenson for the fruitful
discussion of this paper between them when he visited University of Amsterdam in 2025.


\end{document}